\documentclass[a4paper,10pt]{article}
\usepackage{amsmath,amsthm,amssymb}
\usepackage[usenames,dvipsnames,svgnames,table]{xcolor}
\usepackage[plain]{fullpage}
\usepackage{enumerate} 
\usepackage{subcaption}

\usepackage{hyperref}
\hypersetup{
    pdftitle={Diameter of the inversion graph},
    pdfauthor={Frédéric Havet, Florian Hörsch, Clément Rambaud},
    colorlinks,
    linkcolor={RoyalBlue},
    citecolor={RubineRed},
    urlcolor={blue!80!black}
}

\usepackage{tikz}
\usepackage{authblk,enumitem}
\usepackage{thm-restate}

\newtheorem{theorem}{Theorem}[section]
\newtheorem{corollary}[theorem]{Corollary}
\newtheorem{proposition}[theorem]{Proposition}
\newtheorem{lemma}[theorem]{Lemma}
\newtheorem{claim}{Claim}[theorem]
\newtheorem{observation}[theorem]{Observation}

\theoremstyle{definition}

\newtheorem{problem}[theorem]{Problem}
\newtheorem{conjecture}[theorem]{Conjecture}

\newtheorem{cas}{Case}

\newcommand{\ID}{\alpha\mbox{-ID}}

\newcommand{\sreach}{\mathrm{SReach}}

\renewcommand{\leq}{\leqslant}

\renewcommand{\geq}{\geqslant}

\newenvironment{subproof}{\par\noindent {\it Proof}.\ }{\hfill$\lozenge$\par\vspace{11pt}}

\DeclareMathOperator{\rk}{rk}
\DeclareMathOperator{\Inv}{Inv}

\DeclareMathOperator{\inv}{inv}

\DeclareMathOperator{\diam}{diam}
\DeclareMathOperator{\tw}{tw}
\DeclareMathOperator{\dist}{dist}
\DeclareMathOperator{\Ad}{Ad}
\DeclareMathOperator{\Mad}{Mad}
\DeclareMathOperator{\bigO}{\mathcal{O}}

\title{Diameter of the inversion graph}

\author[1]{Fr\'ed\'eric Havet}
\author[2]{Florian Hörsch}
\author[1]{Cl\'ement Rambaud}

\affil[1]{Universit\'e C\^ote d'Azur, CNRS, Inria, I3S, Sophia Antipolis, France}
\affil[2]{CISPA, Saarbrücken, Germany}

\date{}

\begin{document}
\maketitle

\begin{abstract}
In an oriented graph $\vec{G}$, the inversion of a subset $X$ of vertices consists in reversing the orientation of all arcs with both endvertices in $X$.
The inversion graph of a labelled graph $G$, denoted by ${\mathcal{I}}(G)$, is the graph whose vertices are the labelled orientations of $G$ in which two labelled orientations $\vec{G}_1$ and $\vec{G}_2$ of $G$ are adjacent if and only if
there is an inversion $X$ transforming $\vec{G}_1$ into $\vec{G}_2$.
In this paper, we study the inversion diameter of a graph which is the diameter of its inversion graph denoted by $\diam(\mathcal{I}(G))$. We show that the inversion diameter is tied to the star chromatic number, the acyclic chromatic number and the oriented chromatic number. Thus a graph class has bounded inversion diameter if and only if it also has bounded  star chromatic number, acyclic chromatic number and oriented chromatic number.
We give some upper bounds on the inversion diameter of a graph $G$ contained in one of the following graph classes: 
planar graphs ($\diam(\mathcal{I}(G)) \leq 12$), planar graphs of girth 8 ($\diam(\mathcal{I}(G)) \leq 3$),
graphs with maximum degree $\Delta$ ($\diam(\mathcal{I}(G)) \leq 2\Delta -1$), graphs with treewidth at mots $t$ ($\diam(\mathcal{I}(G)) \leq 2t$).

We also show that determining the inversion diameter of a given graph is NP-hard.

\medskip
    \noindent{}{\bf Keywords:}  inversion graph; diameter; orientation; reconfiguration.
\end{abstract}

\section{Introduction}

The concept of inversion was defined by Belkhechine et al. in~\cite{BBBP10} as follows.
In an oriented graph $\vec{G}$, if $X$ is a set of vertices of $\vec{G}$, the {\bf inversion} of $X$ consists in reversing the orientation of all arcs with both endvertices in $X$.
In particular, the reversal of a single arc is an inversion.
Belkhechine et al. in~\cite{BBBP10} studied the {\bf inversion number}, denoted by $\inv(D)$, which is the minimum number of inversions that transform $\vec{G}$ into a directed acyclic graph.
In particular, they proved that for every fixed $k$, determining whether a given tournament $T$ has inversion number at most $k$ is polynomial-time solvable.
In contrast, Bang-Jensen et al.~\cite{BCH} proved that deciding whether a given oriented graph has inversion number~$1$ is NP-complete.
The maximum $\inv(n)$ of the inversion numbers of all oriented graphs of order $n$ has also been investigated. Independently, Aubian et al.~\cite{inversion} and Alon et al.~\cite{APSSW} proved
$n - 2\sqrt{n\log n} \leq \inv(n) \leq n - \lceil \log (n+1) \rceil$.
Recently Duron et al.~\cite{duron2023minimum} studied the minimum number of inversions to make a digraph
$k$-arc-strong or $k$-strong.

In this paper, we study the maximum number of required inversions to transform an orientation of a given graph into another orientation of this graph.
Formally, let $G$ be a graph with vertices labelled $v_1, \dots , v_n$.
The {\bf (labelled) inversion graph} of $G$, denoted by ${\mathcal{I}}(G)$, is the graph whose vertices are the labelled orientations of $G$ in which two labelled orientations $\vec{G}_1$ and $\vec{G}_2$ of $G$ are adjacent if and only if
there is an inversion $X$ transforming $\vec{G}_1$ into $\vec{G}_2$.
The {\bf (labelled) inversion diameter} of $G$ is the diameter of ${\mathcal{I}}(G)$, denote by $\diam(\mathcal{I}(G))$.
We prove various bounds on $\diam(\mathcal{I}(G))$ depending on the structure of $G$.

We first determine the maximum inversion diameter over all graphs on $n$ vertices.

\begin{theorem}\label{thm:diam-upper}
Let $G$ be a graph on $n$ vertices.
Then $\diam ({\mathcal{I}}(G)) \leq n-1$.
\end{theorem}
This bound is tight for complete graphs (Theorem~\ref{thm:multipartite}). 

For sparser graphs, one can expect much better bounds.
Actually, we characterise classes of graphs having bounded inversion diameter.
Indeed, we prove that the inversion diameter is tied\footnote{Two graph parameters $\gamma_1$ and $\gamma_2$ are {\bf tied} if there are two functions $f_1$ and $f_2$ such that $\gamma_1(G) \leq f_1(\gamma_2(G))$ and $\gamma_2(G) \leq f_2(\gamma_1(G))$ for every graph $G$.}
to some well-studied parameters, namely the star chromatic number $\chi_s$, the acyclic chromatic number $\chi_a$,
and the oriented chromatic number $\chi_o$.

Recall that 
a {\bf star-forest} is a forest whose components have diameter at most 2.
The {\bf star chromatic number $\chi_s(G)$} (resp. {\bf acyclic chromatic number $\chi_a(G)$}) of a graph $G$ is the least integer $k$ such that $G$ has a proper $k$-vertex-colouring in which 
any union of two colour classes induces a star-forest (resp. a forest).

The {\bf oriented chromatic number $\chi_o(D)$} of an oriented graph $D$ is the minimum integer $k$ for which there exists a proper colouring $\phi:V(D)\rightarrow [k]$ such that for all $i,j \in [k]$, the arcs linking $\phi^{-1}(i)$ and $\phi^{-1}(j)$ either are all oriented from $\phi^{-1}(i)$ to $\phi^{-1}(j)$ or are all oriented from $\phi^{-1}(j)$ to $\phi^{-1}(i)$. Equivalently, $\chi_o(D)$ is the least integer $k$ such that there exists a homomorphism from $D$ to  a tournament on $k$ vertices. The {\bf oriented chromatic number $\chi_o(G)$} of a graph $G$ is the maximum of $\chi_o(\vec{G})$ over all orientations $\vec{G}$ of $G$.
The star chromatic number, acyclic chromatic number, and oriented chromatic number are known to be tied~\cite{Gru73, RaSo94, KSZ97}, and we add to this list of equivalent parameters the inversion diameter.

\begin{theorem}\label{grclasses}
    Let $\mathcal{G}$ be a class of graphs. The following are equivalent.
    \begin{itemize}
        \item There is a constant $C$ such that $\diam(\mathcal{I}(G)) \leq C$ for every $G \in \mathcal{G}$.
        \item There is a constant $C_s$ such that $\chi_s(G) \leq C_s$ for every $G \in \mathcal{G}$.
        \item There is a constant $C_a$ such that $\chi_a(G) \leq C_a$ for every $G \in \mathcal{G}$.
        \item There is a constant $C_o$ such that $\chi_o(G) \leq C_o$ for every $G \in \mathcal{G}$.
    \end{itemize}
\end{theorem}

As a consequence, graphs of bounded maximum degree, graphs of bounded treewidth, planar graphs, have all bounded inversion diameter.
Since these theoretical bounds are not very accurate, we show more precise upper bounds for these specific classes.
In particular, we show:
\begin{enumerate}[label=\alph*)]
    \item $\diam(\mathcal{I}(G)) \leq 12$ for every planar graph $G$ (Corollary~\ref{cor:planar_12}),
        and there are planar graphs with  inversion diameter at least $5$ (Proposition~\ref{prop:planar5});
    \item $\diam(\mathcal{I}(G)) \leq 3$ for every planar graph $G$ of girth at least $8$ (Theorem~\ref{thm:planar-3}),
        and there are planar graphs of arbitrary large girth with inversion diameter $3$ (Proposition~\ref{prop:planar_even_girth_inversion_3});
    \item $\diam(\mathcal{I}(G)) \leq 2\Delta-1$ for every graph $G$ with maximum degree at most $\Delta \geq 1$ (Theorem~\ref{thm:Delta}),
        and $\diam(\mathcal{I}(G))=\Delta(G)$ for every complete graph $G$ (Theorem~\ref{thm:multipartite});
    \item $\diam(\mathcal{I}(G)) \leq 2t$ for every graph $G$ of treewidth at most $t$ (Theorem~\ref{thm:Mtw2t}),
        and for every $t \geq 2$, there are graphs of treewidth $t$ with  inversion diameter $t+2$ (Theorem~\ref{thm:lower_bound_for_tw}).
\end{enumerate}

We also show that determining the inversion diameter of a given graph is NP-hard, and that determining whether two given orientations of a given bipartite planar graph are at distance at most $2$ in its inversion graph is NP-complete (Corollary~\ref{cor:NP-diam}).

\paragraph*{Structure of the paper}

In Section~\ref{sec:multipartite}, we give an algebraic characterisation of the inversion diameter that will be used throughout the paper, 
and we apply it to determine the inversion diameter of balanced complete multipartite graphs, including complete graph.
In Section~\ref{sec:strong_degeneracy}, we show that despite the fact that there are $2$-degenerate graphs with arbitrary high inversion diameter, a slightly stronger notion of degeneracy 
is an upper bound on the inversion diameter. This fact is one of our main tools to upper bound the inversion diameter.
In Section~\ref{sec:carac}, we show that the inversion diameter is tied to the star chromatic number, the acyclic chromatic number, and the oriented chromatic number.
We thus characterise the classes of graphs having inversion diameter bounded by a constant.
In Section~\ref{sec:better}, we improve the bounds on the inversion diameter as functions of $\chi_s$ and $\chi_a$ given in the previous section.
In Section~\ref{sec:Delta}, we show that $\diam(\mathcal{I}(G)) \leq 2\Delta(G)-1$ for every graph $G$ with at least one edge, and we improve this bound for $\Delta(G)\leq 3$.
In Section~\ref{sec:minor_closed}, we show various bounds for the maximum inversion diameter of graphs in some minor-closed classes.
In particular, we show the above-mentioned results on planar graphs (items a) and b)) and bounded-treewidth graph (item c)).
In Section~\ref{sec:complexity}, we show that computing distances in $\mathcal{I}(G)$ is hard. More precisely, for every $k \geq 2$, it is NP-hard to determine if two given orientations are at distance at most $k$, and it is NP-hard to determine if $\diam(\mathcal{I}(G)) \leq k$.
This is a consequence of Theorem~\ref{thm:diam_L_G_subdivided_once} which expresses the inversion diameter of $G^{(1)}$, the graph obtained from a graph $G$ by subdividing every edge once, in term of $\chi(G)$, which is well-known to be hard to compute.

\paragraph*{Notation}

For every positive integer $k$, we denote by $[k]$ the set $\{1,\dots,k\}$.
We denote by $\log(\cdot)$ the logarithm in base $2$.
Given a graph $G$, we denote by $V(G)$ its vertex set and by $E(G)$ its edge set. Graphs do not contain loops or multiple edges, otherwise, we speak of {\it multigraphs}.
Similarly, given a digraph $D$ we denote by $V(D)$ its vertex set and by $A(D)$ its arc set.
If $X$ is a set of vertices in $G$ (resp. $D$), we denote by $G\langle X \rangle$ (resp. $D\langle X \rangle$) the subgraph (resp. subdigraph) induced by $X$.

We denote by $\delta(G)$ the minimum degree of $G$ and by $\delta^*(G)$ the {\bf degeneracy} of $G$, which is $\max\{\delta(H) \mid H \mbox{ subgraph of } G\}$.
Equivalently, $\delta^*(G) \leq d$ if and only if there is an ordering $(v_1, \dots, v_n)$ of $V(G)$ such that for every $i \in [n]$, $v_i$ has at most $d$ neighbours in $\{v_j \mid j<i\}$.
We also denote by $\Ad(G)$ the average degree of $G$, that is $\frac{2|E(G)|}{|V(G)|}$, and by $\Mad(G)$ the maximum over all subgraphs $H$ of $G$ with at least one vertex of $\Ad(H)$.
Note that $\Mad(G)/2 \leq \delta^*(G) \leq \Mad(G)$.

More generally, we use standard notations in graph theory, and we refer the reader to~\cite{bondymurty} for undefined notion or notation,
such as the parameters $\delta, \Delta, \chi$.

Consider now an oriented graph $\vec{G}$ and some $X \subseteq V(\vec{G})$.
We denote by $\Inv(\vec{G};X)$ the result of the inversion of $X$ in $\vec{G}$, and more generally, if $X_1, \dots, X_t$ are sets of vertices in $\vec{G}$,
$\Inv(\vec{G};(X_i)_{i \in [t]})$ is the oriented graph resulting from the inversion of $X_1, X_2, \dots, X_t$. Note that $\Inv(\vec{G};(X_i)_{i \in [t]})$ does not depends on the 
ordering of $X_1, \dots, X_t$, since an arc $uv$ is reversed if and only if $|\{i \in [t] \mid u,v \in X_i\}|$ is odd.

We will often use  arguments from linear algebra. The vectors and matrices considered will always be over the field $\mathbb{F}_2$ with two elements $0$ and $1$.
Given a matrix $M$, we denote by $\rk(M)$ its rank over $\mathbb{F}_2$.
We denote by $\mathbf{0}$ the vector in $\mathbb{F}_2^d$ constant to $0$, for $d$ a positive integer that will be clear from the context.
Moreover, given two vectors $\mathbf{x} = (x_1, \dots,x_d),\mathbf{y} = (y_1, \dots, y_d)$ of same dimension $d$,
we denote by $\mathbf{x} \cdot \mathbf{y}$ their scalar product, that is $\sum_{i=1}^d x_i y_i$.

\section{Some tools and complete multipartite graphs}\label{sec:multipartite}

In this section, we introduce some tools to compute $\diam(\mathcal{I}(G))$, and we apply them on some highly structured graphs.

Let $\vec{G}_1$ and $\vec{G}_2$ be two orientations of a graph $G$.
If an edge $e$ has the same orientation in $\vec{G}_1$ and $\vec{G}_2$, we say that $\vec{G}_1$ and $\vec{G}_2$ {\bf agree} on $e$;
otherwise we say that they {\bf disagree} on $e$. 
We denote by $E_{=}$ the set of edges of $G$ on which $\vec{G}_1$ and $\vec{G}_2$ agree 
and by $E_{\neq}$ the set of edges of $G$ on which $\vec{G}_1$ and $\vec{G}_2$ disagree.

Theorem~\ref{thm:diam-upper} follows directly from the following lemma and an immediate induction.

\begin{lemma}\label{lem:rec-easy}
    Let $G$ be a graph and $v$ a vertex in $G$.
    Then $\diam ({\mathcal{I}}(G)) \leq \diam ({\mathcal{I}}(G-v)) +1$.
\end{lemma}
\begin{proof}
    Let $\vec{G}_1$ and $\vec{G}_2$ be two orientations of $G$.  
    Let $X'$ be the set of neighbours $w$ of $v$ such that $\vec{G}_1$ and $\vec{G}_2$ disagree on $vw$.
    Set $X=X'\cup \{v\}$.
    Note that $v$ has the same out-neighbours and the same in-neighbours in $\Inv(\vec{G}_1 ; X)$ and $\vec{G}_2$.
    Now let $\cal{X}$ be a collection of $\diam({\mathcal{I}}(G-v))$ subsets of $V(G-v)$ such that $\Inv(\Inv(\vec{G}_1;X) -v ;{\cal X}) = \vec{G}_2 - v$. We obtain $\Inv(\vec{G}_1 ;{\cal X}\cup \{X\}) = \vec{G}_2$ and hence the statement follows.
\end{proof}

Now we give a fundamental observation that allows us to use tools from linear algebra to bound $\diam(\mathcal{I}(G))$.
\begin{observation}\label{obs:characterization_with_vectors}
    For every graph $G$, for every positive integer $t$, the following are equivalent.
    \begin{enumerate}
        \item $\diam(\mathcal{I}(G)) \leq t$.
        \item For every function $\pi \colon  E(G) \to \mathbb{F}_2$, there is a family $(\mathbf{u})_{u \in V(G)}$ of vectors in $\mathbb{F}_2^t$ such that
        \[
        \pi(uv) = \mathbf{u} \cdot \mathbf{v}
        \]
        for every edge $uv \in E(G)$.
    \end{enumerate}
\end{observation}

\begin{proof}
    First suppose that $\diam(\mathcal{I}(G)) \leq t$, and let $\pi \colon  E(G) \to \mathbb{F}_2$.
    Let $\vec{G}_1, \vec{G}_2$ be two orientations of $G$ such that, for every edge $uv \in E(G)$, $\vec{G}_1$ and $\vec{G}_2$ agree on $uv$ if and only if $\pi(uv)=0$.
    By assumption, there are sets $X_1, \dots, X_t \subseteq V(G)$ such that $\vec{G}_1 = \Inv(\vec{G}_2; X_1, \dots, X_t)$.
    For every vertex $u \in V(G)$, we define $\mathbf{u}$ to be the vector in $\mathbb{F}_2^t$ whose $i^{\text{th}}$ coordinate is $1$ if and only if $u \in X_i$, for every $i \in [t]$.
    Now observe that for every edge $uv \in E(G)$, $uv$ is reversed by the inversions $X_1, \dots, X_t$ if and only if $\mathbf{u} \cdot \mathbf{v}=1$.
    Hence $\mathbf{u} \cdot \mathbf{v} = \pi(uv)$.

    Reciprocally, given two orientations $\vec{G}_1$ and $\vec{G}_2$ of $G$, for every edge $uv \in E(G)$, let $\pi(uv)=0$ if and only if $\vec{G}_1$ and $\vec{G}_2$ agree on $uv$.
    By assumption, there is a family $(\mathbf{u})_{u \in V(G)}$ of vectors in $\mathbb{F}_2^t$ such that $\pi(uv) = \mathbf{u} \cdot \mathbf{v}$ for every edge $uv \in E(G)$.
    We define $X_i = \{u \in V(G) \mid \mathbf{u}_i = 1\}$ for every $i \in [t]$.
    Then for every edge $uv \in E(G)$, $uv$ is reversed by the inversions $X_1, \dots, X_t$ if and only if $\mathbf{u} \cdot \mathbf{v}=1$, and so $\vec{G}_2 = \Inv(\vec{G}_1; X_1, \dots, X_t)$.
    This proves that $\diam(\mathcal{I}(G)) \leq t$.
\end{proof}

We now fully determine the inversion diameter of 
balanced complete multipartite graphs. 
The purpose of this result is two-fold. Firstly, its proof may serve as a warm-up exercise for the following sections. 
Secondly, it gives in particular the inversion diameter of complete graphs, which serve as tightness examples in several results in the following sections.
We denote by $K_r[\overline{K_t}]$ the complete $r$-partite graph in which every part has size $t$. 

\begin{theorem}\label{thm:multipartite}
    For all positive integers $r$ and $t$, we have $\diam\big(\mathcal{I}(K_r[\overline{K_t}])\big) = (r-1)t$.
\end{theorem}

\begin{proof}
    Let the graph $G$ be defined by $V(G) = [r] \times [t]$,  and $E(G) = \{(i,j)(i',j') \mid i \neq i'\}$. Observe that $G$ is isomorphic to $K_r[\overline{K_t}]$.
    By repeatedly applying Lemma~\ref{lem:rec-easy} for all vertices in $V(G)\setminus (\{1\}\times [t])$ and using the fact that the inversion diameter of edgeless graphs is $0$,
    we have $\diam(\mathcal{I}(G)) \leq (r-1) t + \diam(\mathcal{I}(\overline{K}_t)) = (r-1)t$.
    Let us now prove that $\diam(\mathcal{I}(G)) \geq (r-1)t$.

    Consider the function $\pi \colon E(G) \to \mathbb{F}_2$ defined by 
    $\pi((i,j)(i',j'))=1$ if and only if $|i'-i|=1$ and $j=j'$,
    for all $i,i' \in [r]$ and $j,j' \in [t]$ with $i \neq i'$.
    Suppose that there is a family $(\mathbf{v}_{i,j})_{i \in [r], j \in [t]}$ of vectors in  $\mathbb{F}_2^k$
    such that $\pi((i,j)(i',j'))= \mathbf{v}_{i,j} \cdot \mathbf{v}_{i',j'}$ for every $i,i' \in [r]$ and $j,j' \in [t]$ with $i \neq i'$.
    Let $U$ be the matrix whose column $(i,j)$ is the vector $\mathbf{v}_{i,j}$ (in the lexicographic order $(1,1),(1,2), \dots, (r,t)$). 
    Then $U^\top \cdot U$ is of the form
    \[
    U^\top \cdot U = 
    \left( 
    \begin{array}{ccccc}
        \star  & I_t    & 0      & \dots   & 0     \\
        I_t    & \star  & I_t    & \ddots  & \vdots \\
        0      & I_t    & \ddots &  \ddots &   0   \\
        \vdots & \ddots & \ddots &  \star  &  I_t  \\
        0      & \dots  & 0      &   I_t   & \star \\
    \end{array}
    \right)
    \] 
    As $U^\top \cdot U$ contains an $(r-1)t\times(r-1)t$ upper triangular matrix whose diagonal is constant to $1$, we obtain that $U^\top \cdot U$ has rank at least $(r-1)t$.
    It follows that $k \geq \rk(U) \geq \rk(U^\top \cdot U) \geq (r-1)t$. Using Observation~\ref{obs:characterization_with_vectors}, this proves that $\diam(\mathcal{I}(G)) \geq (r-1)t$.
\end{proof}




\section{Bounds in terms of degeneracy and strong degeneracy}\label{sec:strong_degeneracy}
In this section, we study the relationship between the inversion diameter of a graph $G$ and its degeneracy. 
Note that the degeneracy of a graph is strongly related to its maximum average degree, as $\delta^*(G)\leq \Mad(G)\leq 2 \delta^*(G)$ holds for every graph $G$. 
In Proposition~\ref{prop:density}, we show that bounded degeneracy (or maximum average degree) alone is not enough to bound the inversion diameter of a graph. 
On the other hand, in Theorem~\ref{theorem:high_ad_implies_high_diam}, we show that a bounded inversion diameter implies bounded degeneracy (and bounded maximum average degree). 
Afterwards, we introduce a new notion of degeneracy, which we call strong degeneracy. 
In Theorem~\ref{thm:generic_greedy_argument}, we show that, in contrast to the classic notion of degeneracy, bounded strong degeneracy guarantees a bounded inversion diameter. 
Theorem~\ref{thm:generic_greedy_argument} will serve as a valuable tool in several parts of this article. 
As a first application of Theorem~\ref{thm:generic_greedy_argument}, we prove in Corollary~\ref{cor:dqewdrqwde} that in graphs of bounded degeneracy, the maximum inversion diameter is logarithmic in the maximum degree.

We start by proving that the inversion diameter of a graph $G$ cannot be upper bounded by a function of $\delta^*(G)$ or $\Mad(G)$.

\begin{proposition}\label{prop:density}
    For every integer $\ell$, there exists a graph $G$ with  $\delta^*(G) \leq 2$  and $\diam(\mathcal{I}(G)) \geq \ell$.
\end{proposition}
\begin{proof}
    Let $G$ be the graph obtained from $K_{2^{\ell-1}+1}$ by subdividing every edge $e$ once, 
    thereby creating a vertex $x_{e}$.
    We call the vertices of $K_{2^{\ell-1}+1}$ the original vertices of $G$.
    One can easily check that $\delta^*(G) \leq 2$.
    
    Let $\vec{G}_1$ be an orientation of $G$ where $d_{\vec{G}_1}^-(x_e)=1$ for all $e \in E(K_{2^{\ell-1}+1})$ and let $\vec{G}_2$ be an orientation of $G$ where $d_{\vec{G}_2}^-(x_e)=2$ for all $e \in E(K_{2^{\ell-1}+1})$.
    Suppose that there exist $X_1, \dots X_t \subseteq V(G)$ such that $\Inv(\vec{G}_1;X_1, \dots, X_t) = \vec{G}_2$
    for some $t < \ell$.
    By the pigeon-hole principle, there exist two original vertices $u$ and $v$ such that $\{i \mid u \in X_i\} = \{i \mid v \in X_i\}$. It follows that $d^-_{\Inv(\vec{G}_1;X_1, \dots, X_t)}(x_{uv})=1$. 
    This is a contradiction to the assumption that $\Inv(\vec{G}_1;X_1, \dots, X_t) = \vec{G}_2$.
    Hence $\diam(\mathcal{I}(G)) \geq \ell$.
\end{proof}

Conversely, the following result shows that the inversion diameter grows with the maximum average degree and the degeneracy.

\begin{theorem}\label{theorem:high_ad_implies_high_diam}
    $\diam(\mathcal{I}(G)) \geq \Mad(G)/2$ for every graph $G$.
\end{theorem}

\begin{proof}
    Set $M=\Mad(G)$.
    Let $H$ be a subgraph of $G$ such that $\Ad(H) =M$.  
    Let $n$ be the number of vertices of $H$.
    The number of edges of $H$ is $\frac{1}{2}Mn$.
    Hence $H$ has $2^{Mn/2}$ orientations. 
    On the other hand, starting from an arbitrary orientation $\vec{G}$ of $G$, we have $2^n$ possible choices of subsets of $V(G)$ to invert. 
    Hence a simple induction shows that, for any positive integer $t$, we have that  at most $2^{nt}$ orientations can be obtained from $\vec{G}$ through $t$ inversions.
    This yields that if $t = \diam(\mathcal{I}(H))$, then $H$ has at most $2^{nt}$ orientations. Thus $t \geq M/2$.
    
    Finally, since $H$ is a subgraph of $G$, we have $\diam(\mathcal{I}(G)) \geq \diam(\mathcal{I}(H)) \geq M/2$.
\end{proof}

However, a stronger version of degeneracy permits to bound $\diam(\mathcal{I}(G))$.
First, we need a few definitions.

Let $G$ be a graph and let $<$ be a total ordering on $V(G)$.
For every pair $u,u'$ of vertices in $G$, let $N_{<u'}(u) = \{v \in N(u) \mid v<u'\}$ and $N_{>u'}(u) = \{v \in N(u) \mid v>u'\}$.
We simply write $N_{<}(u)$ for $N_{<u}(u)$ and $N_{>}(u)$ for $N_{>u}(u)$.
The ordering $<$ is {\bf $t$-strong} if for every $u\in V(G)$
\begin{itemize}
 \item $|N_<(u)| + 
        \log \left(|\{X \subseteq V(G) \mid \exists v \in N_>(u), X \subseteq N_{<u}(v)\}|\right)
    < t$ if $N_>(u) \neq \emptyset$, and
\item $|N_<(u)| \leq t$ otherwise.
 \end{itemize}


A graph is {\bf strongly $t$-degenerate} if there is a $t$-strong ordering of its vertices.
We now show that this stronger version of degeneracy is sufficient to provide a bound for the inversion diameter.

\begin{theorem}\label{thm:generic_greedy_argument}
    Let $G$ be a graph and let $t$ be a positive integer.
    If $G$ is strongly $t$-degenerate, then  $\diam(\mathcal{I}(G)) \leq t$.
\end{theorem}

\begin{proof}
    Consider a function $\pi \colon E(G) \to \mathbb{F}_2$.
    By Observation~\ref{obs:characterization_with_vectors},
    it is enough to show the existence of a family $(\mathbf{u})_{u \in V(G)}$
    of vectors in $\mathbb{F}_2^t$ such that $\mathbf{u} \cdot \mathbf{v} = \pi(uv)$ for every edge $uv$ of $G$.
    Let $<$ be a $t$-strong ordering of $V(G)$ and let $(v_1, \dots, v_n)$  be the vertices of $V(G)$ as they appear in this order.
    We iteratively construct vectors $\mathbf{v}_1, \dots, \mathbf{v}_n \in \mathbb{F}_2^t$ maintaining the following invariant:
    for $i \in [n]$
    \begin{enumerate}
        \item for every $k \in \{i+1,\ldots,n\}$, the vectors $(\mathbf{v}_j)_{ v_j \in N_{<v_i}(v_k)}$ are linearly independent, and
        \item $\mathbf{v}_j \cdot \mathbf{v}_k = \pi(v_j v_k)$ for every edge $v_j v_k$ of $G$ with $j,k\leq i$.
    \end{enumerate}
    For $i=n$, this will imply that $\diam(\mathcal{I}(G)) \leq t$ by Observation~\ref{obs:characterization_with_vectors}.

    For $i=1$, we take $\mathbf{v}_1 \in \mathbb{F}_2^t \setminus\{\mathbf{0}\}$ arbitrarily.
    Now suppose $i>1$ and that $\mathbf{v}_1, \dots, \mathbf{v}_{i-1}$ have already been constructed.
    Since the vectors $(\mathbf{v}_j)_{v_j \in N_<(v_i)}$ are linearly independent, the set of vectors $\mathbf{x} \in \mathbb{F}_2^t$ satisfying $\mathbf{v}_j \cdot \mathbf{x} = \pi(v_i v_j)$ for every $v_j \in N_<(v_i)$ forms an affine space $A$ of dimension at least $t - |N_<(v_i)|$. To satisfy the second condition, we must choose $\mathbf{v}_i$ in $A$. Note that $|A|\geq 2^{t - |N_<(v_i)|}\geq 1$ because $<$ is $t$-strong.
    Observe that if $N_>(v_i)=\emptyset$, then the first condition trivially holds, and so we can choose any vector in $A$ for $\mathbf{v}_i$.

    Henceforth, we may assume that $N_>(v_i)\neq \emptyset$.
    For every $v_j\in N_>(v_i)$, the first condition of the invariant says that $\mathbf{v}_i$ must be outside the vector space $U_j$ spanned by $(\mathbf{v_k})_{v_k \in N(v_j), k<i}$.
    Thus it is enough to take $\mathbf{v}_i$ not in $U = \bigcup_{v_j \in N_>(v_i)} U_j$.
    Now observe that for every $\mathbf{y} \in U$, there is a vertex $v_j \in N_>(v_i)$ and a set $X \subseteq \{v_k \in N(v_j) \mid k<i\}$ such that $\mathbf{y} = \sum_{x \in X} \mathbf{x}$.
    Hence $|U| \leq |\{X \subseteq V(G) \mid \exists v_j \in N_>(v_i), X \subseteq N_{<v_i}(v_j) \}|$. 
    Since $<$ is $t$-strong, we obtain that $|A|=2^{t-|N_<(v_i)|} > |\{X \subseteq V(G) \mid \exists v_j \in N_>(v_i), X \subseteq N_<(v_j)\}| \geq |U|$.  
    We deduce that $A \setminus U \neq \emptyset$ and there is a vector $\mathbf{v}_i$ satisfying the invariant.
    The statement now follows for $i=n$.
\end{proof}


As a first consequence of this proposition, we prove that graphs of bounded degeneracy have inversion diameter at most logarithmic in the maximum degree.

\begin{corollary}\label{cor:dqewdrqwde}
     $\diam(\mathcal{I}(G)) \leq 2\delta^*(G) - 1 + \lceil\log(\Delta(G))\rceil$ for every graph $G$ with at least one edge.
\end{corollary}

\begin{proof}
    Let $d = \delta^*(G)$, let $t=2\delta^*(G)-1 + \lceil\log(\Delta(G))\rceil$, and 
    let $<$ be an ordering of $V(G)$ witnessing the fact that $\delta^*(G) \leq d$.
    If $\Delta(G)=1$, then one can easily check that $\diam(\mathcal{I}(G)) \leq 1 \leq 2d-1$.
    Now assume $\Delta(G) \geq 2$.
    For every vertex $u \in V(G)$, we have $|N_<(u)| \leq d$ by definition.
    Moreover, if $N_>(u) \neq \emptyset$ then $|\{X \subseteq V(G) \mid \exists v \in N_>(u), X \subseteq N_{<u}(v)\}| \leq 1 + \sum_{v \in N(u), u<v} (2^{|N_{<u}(v)|}-1) \leq 1+\Delta(G)(2^{d-1}-1)<\Delta(G)2^{d-1}$ (since $\Delta(G)\geq 2$).
    Finally, since $t\geq 2d-1 + \log(\Delta(G))$, the ordering $<$ is $t$-strong, and the result follows from Theorem~\ref{thm:generic_greedy_argument}.    
\end{proof}

\section{Tying the inversion diameter}\label{sec:carac}

The aim of this section is to prove that the inversion diameter of a graph behaves asymptotically in the same way as three other graph parameters which have been studied before. More precisely, we show the following result relating the inversion diameter of a graph to its star chromatic number, its acyclic chromatic number, and its oriented chromatic number.

\begin{theorem}\label{thm:tied}
The inversion diameter is tied to the star chromatic number, the acyclic chromatic number, and the oriented chromatic number.
\end{theorem}

Observe that Theorem \ref{thm:tied} directly implies Theorem \ref{grclasses}. In order to prove Theorem~\ref{thm:tied}, we first give some simple preliminary results on the inversion diameter of forests in Section~\ref{sec:for}. 
These observations will also be reused in some later sections. After, in Section~\ref{sec:relate}, we give the main proof of Theorem~\ref{thm:tied}.

\subsection{Forests}\label{sec:for}

\begin{theorem}\label{thm:L-arbre}
If $F$ is a forest, then $\diam ({\mathcal{I}}(F)) \leq 2$.  
\end{theorem}

\begin{proof}
    Let $\vec{F}_1$ and $\vec{F}_2$ be two orientations of a forest $F$. Recall that $E_{\neq}$ is the set of edges on which $\vec{F}_1$ and $\vec{F}_2$ disagree.
    Consider the graph $H_{\neq}$ whose vertices are connected components of $(V(F),E_{\neq})$ ,
    and where two such connected components $C$ and $C'$ are connected if and only if
    there is an edge $uv$ in $E(F)$ with $u \in C$ and $v \in C'$.
    Observe that $H_{\neq}$ is a minor of $F$, and so is a forest. In particular,
    $H_{\neq}$ is bipartite. Let $(X'_1, X'_2)$ be a bipartition of $H_{\neq}$ and let $(X_1,X_2)$ the partition of $V(F)$ where a vertex is contained in $X_i$ if it is contained in a component of $H_{\neq}$ that belongs to $X_i'$ for $i=1,2$.
    For every edge $uv$ in $E(F)$, we have $uv \in E_{\neq}$
    if and only if $\{u,v\} \subseteq X_1$ or $\{u,v\} \subseteq X_2$.
    Thus $\Inv(\vec{F}_1;(X_1,X_2))=\vec{F}_2$.
\end{proof}

We next give a simple characterisation of the graphs whose inversion diameter is at most $1$.
A {\bf star forest} is a graph whose connected components are all isomorphic to a star (including $K_1$ and $K_2$).
\begin{proposition}\label{prop:diam1}
    Let $G$ be a graph.
   \begin{itemize}
   \item[(i)]  $\diam(\mathcal{I}(G)) =0$ if and only if $G$ has no edges.
   \item[(ii)]  $\diam(\mathcal{I}(G))\leq 1$ if and only if $G$ is a star forest.
   \end{itemize}
\end{proposition}

\begin{proof}
    (i) holds trivially.
    Let us prove (ii).
    It is easy to see that if $G$ is a star-forest, then $\diam(\mathcal{I}(G)) \leq 1$.
    Assume now that $G$ is not a star forest.
    Then $F$ has a subgraph which is isomorphic to either the complete graph on three vertices $K_3$, or the path on four vertices $P_4$.
    Now $\diam(\mathcal{I}(K_3))=2$ by Theorem~\ref{thm:multipartite}, and $\diam(\mathcal{I}(P_4))=2$ since 
    two orientations of $P_4$ agreeing only on its middle edge are at distance $2$.
    Thus
    $\diam(\mathcal{I}(G)) \geq \min \{\diam(\mathcal{I}(K_3)),\diam(\mathcal{I}(P_4))\}  = 2$. 
\end{proof}

We now use Theorem~\ref{thm:L-arbre} and Proposition~\ref{prop:diam1} to completely determine the inversion diameter of forests.
\begin{corollary}\label{cor:forest}
    If $F$ is a forest, then
    \[
    \diam(\mathcal{I}(F)) = 
    \left\{
    \begin{array}{l l}
        0 & \text{ if }F\text{ has no edge},\\
        1 & \text{ if }F\text{ is a star forest with at least one edge},\\
        2 & \text{ otherwise.}
    \end{array}
    \right.
    \]
\end{corollary}

\subsection{Proof of Theorem~\ref{thm:tied}}\label{sec:relate}

In this section, we prove that the inversion diameter of a graph is tied to its star chromatic number, its acyclic chromatic number, and its oriented chromatic number. We first show in Corollary~\ref{cor:chi_s} that the inversion diameter of a graph is bounded by functions of its star chromatic number and its acyclic chromatic number. We then show in Theorem~\ref{thm:lower-chi_o} that the oriented chromatic number of a graph is bounded by a function of its inversion diameter. Together with some previous results showing that the star chromatic number, the acyclic chromatic number, and the oriented chromatic number are tied, we obtain Theorem~\ref{thm:tied}.

For the first part, it turns out to be convenient to introduce the following slightly more general graph parameter.

An {\bf $\alpha$-ID $k$-colouring}  of a graph $G$ is a $k$-colouring such that the subgraph induced by any two colour classes has inversion diameter at most $\alpha$
A graph is {\bf $\alpha$-ID $k$-colourable} if it admits a  $\alpha$-ID $k$-colouring and
the {\bf $\alpha$-ID chromatic number} of a graph $G$, denoted by $\ID(G)$, is the smallest integer $k$ such that
$G$ is $\alpha$-ID $k$-colourable.

Note that an $\alpha$-ID $k$-colouring is not necessarily proper. The following result shows that the inversion diameter of graphs admitting a proper $\alpha$-ID $k$-colouring is bounded by a function depending only on $\alpha$ and $k$.

\begin{lemma}\label{lem:H-col}
    Let $G$ be a graph and $\alpha$ and $k$ positive integers such that $G$ admits a proper $\alpha$-ID $k$-colouring. Then 
    \[\diam(\mathcal{I}(G)) \leq \alpha \binom{k}{2}.
    \]
\end{lemma}

\begin{proof}
    Let $\vec{G}_1$ and $\vec{G}_2$ be two orientations of $G$.
    Assume that $G$ admits a proper $\alpha$-ID $k$-colouring $c \colon V(G)\rightarrow [k].$ For every $i \in [k]$, let $S_i=c^{-1}(i)$.
    For every $i,j$ in $[k]$ such that $i<j$, let $V_{i,j} = S_i\cup S_j$.  By assumption,  there exists a family ${\cal X}_{i,j}$ of $\alpha$ subsets of $V_{i,j}$ whose inversion transforms
    $\vec{G}_1\langle V_{i,j}\rangle$ into $\vec{G}_2\langle V_{i,j}\rangle$. Let $\mathcal{X} = \bigcup_{1\leq i < j\leq k} {\cal X}_{i,j}$.
    We claim that 
    $\Inv(\vec{G}_1,\mathcal{X}) = \vec{G}_2$.
    
    To see that, consider an arc $uv$ of $\vec{G}_1$, and let $i=c(u)$ and $j=c(v)$.
    Note that $i\neq j$ because $c$ is a proper colouring.
    Then $uv$ is reversed only when inverting sets of ${\cal X}_{i,j}$.
    As inverting ${\cal X}_{i,j}$ transforms $\vec{G}_1\langle V_{i,j}\rangle$ into $\vec{G}_2\langle V_{i,j}\rangle$, it follows that 
    $uv$ has been inverted if and only if it belongs to $E_{\neq}$.
    In other words, we obtain $\Inv(\vec{G}_1;\mathcal{X}) = \vec{G}_2$.
\end{proof}

Using Lemma~\ref{lem:H-col} and the fact that star-forests and forests have inversion diameter at most $1$ and $2$ respectively by Corollary~\ref{cor:forest}, we deduce the following.

\begin{corollary}\label{cor:chi_s}\label{cor:L-acyclic}
Let $G$ be a graph. Then 
$\displaystyle 
\diam(\mathcal{I}(G)) \leq \binom{\chi_s(G)}{2}
$  and
    $\displaystyle 
\diam(\mathcal{I}(G)) \leq 2\binom{\chi_a(G)}{2}
$.
\end{corollary}

We now show that the oriented chromatic number of a graph is upper bounded by a function of its inversion diameter.

\begin{theorem}\label{thm:lower-chi_o}
    Let $G$ be a graph with inversion diameter $\ell$. Then $\chi_o(G) \leq (2\ell+1)2^\ell$.
\end{theorem}

\begin{proof}
    Let $G$ be a graph with inversion diameter $\ell$.
    Let $\vec{G}_1$ be an orientation of $G$. We will show that $\vec{G}_1$ has a homomorphism
    to an orientation of $K_{(2\ell+1)2^\ell}$.
    
    Since $\chi(G)\leq \Mad(G)+1$ and $\Mad(G) \leq 2\ell$ by Theorem~\ref{theorem:high_ad_implies_high_diam}, $G$ has a proper $(2\ell+1)$-colouring 
    $\phi: V(G) \to [2\ell+1]$.
    Let $\vec{G}_0$ be the orientation of $G$ such that $uv \in A(\vec{G}_0)$ if and only if
    $\phi(u)<\phi(v)$.
    By definition of $\ell$, there exists a sequence $X_1, \dots, X_\ell$ of subsets of $V(G)$
    such that $\Inv(\vec{G}_0, \{X_1, \dots, X_\ell\}) = \vec{G}_1$.
    We define the colouring $\psi$ as follows:  for every vertex $v$ in $G$, $\psi(v) = (\phi(v) , \{i \in [\ell], v \in X_i\})$.
    Clearly, this is a proper $(2\ell+1)2^\ell$-colouring of $G$.
   Further, for $c \in [2 \ell+1]$ and $I \subseteq [\ell]$, let $V_{(c,I)}$ be the set of vertices coloured $(c,I)$. 
    We claim that, in $\vec{G}_1$, all the arcs between $V_{(c,I)}$ and $V_{(c',I')}$ are either all oriented towards  $V_{(c,I)}$ or all oriented towards  $V_{(c',I')}$ for all $c,c' \in [2 \ell+1]$ and $I,I' \subseteq [\ell]$.
    To see that, observe that those arcs are in the same direction in 
    $\vec{G}_0$, that is from $\phi^{-1}(c)$ to $\phi^{-1}(c')$ if $c<c'$, and that they are always reversed together.
    
    Therefore $\psi$ is an oriented colouring of $\vec{G}_1$. As  $\vec{G}_1$ was arbitrary, we get $\chi_o(G) \leq (2\ell+1)2^\ell$.
\end{proof}

Corollary~\ref{cor:chi_s}, Theorem~\ref{thm:lower-chi_o}, and the fact that
$\chi_s$, $\chi_a$, and $\chi_o$ are tied~\cite{Gru73, RaSo94, KSZ97} imply Theorem~\ref{thm:tied}.

\section{Better bounds in terms of \texorpdfstring{$\chi_s$, $\chi_a$ and $\chi_o$}{chis, chia and chio}}\label{sec:better}

The objective of this section is to quantitatively improve on some of the results obtained in Section~\ref{sec:carac}. In Theorem~\ref{thm:col-acy-gen}, we give an improvement of Lemma~\ref{lem:H-col} in the case that $\alpha \geq 7$. An application of this new result immediately gives an improvement on Corollary~\ref{cor:chi_s}, hence yielding a better understanding of the relationship of the inversion diameter and the parameters $\chi_s(G)$ and $\chi_a(G)$. Further, in Theorem~\ref{thm:diam_L_leq_chi_o_square}, we give a better upper bound on the inversion diameter in terms of $\chi_o$.


For convenience, for a graph parameter $\gamma$ and an integer $k$, we define $M(\gamma \leq k)$ as the maximum inversion diameter over all graphs $G$ such that $\gamma(G) \leq k$ if such a maximum exists, and $+\infty$ otherwise.

\medskip

The main technical contribution for the proof of Theorem~\ref{thm:col-acy-gen} is contained in the following lemma. 
Roughly speaking, we improve on Lemma~\ref{lem:H-col}, by executing simultaneously some of the operations which were executed consecutively before.
\begin{lemma}\label{lem:recur-ID}
$M (\ID\leq  t)\leq  M (\ID\leq  \lceil\frac{t}{2}\rceil)+\alpha \lceil\frac{t}{2}\rceil\lfloor\frac{t}{2}\rfloor$ for every nonnegative integer $t$.   
\end{lemma}
\begin{proof}
Let $G$ be a graph with $\ID(G) \leq t$.   
    Let $\vec{G_1}$ and $\vec{G_2}$ be two orientations of $G$ that are at distance $\diam ({\mathcal{I}}(G))$ in ${\mathcal{I}}(G)$. By assumption, there exists a partition $(S_1,\ldots,S_t)$ of $V(G)$ such that $\diam(\mathcal{I}(G\langle S_i\cup S_j\rangle))\leq \alpha$ holds for every $\{i,j\}\subseteq [t]$. Let $H_1=G\langle S_1\cup\ldots \cup S_{\lceil\frac{t}{2}\rceil}\rangle$ and $H_2=G\langle S_{\lceil\frac{t}{2}\rceil+1}\cup \ldots \cup S_t\rangle$. 
    Clearly, $\ID(H_1) \leq \lceil\frac{t}{2}\rceil$ for all $i\in [2]$. Set $p= M (\ID\leq  \lceil\frac{t}{2}\rceil)$.
    Thus, for each $i\in [2]$, there exists a collection of sets $X_1^{i},\ldots,X_p^{i}\subseteq V(H_i)$ such that $\Inv(\vec{G_1}\langle V(H_i)\rangle;(X^{i}_j)_{j \in [p]})=\vec{G_2}\langle V(H_i)\rangle$. 
    For all $j\in [p]$, we set $X_j=X_j^1\cup X_j^2$.
    Let $D_0 = \Inv(\vec{G_1};(X_i)_{i \in [p]})$.
    Observe that $D_0$ and $\vec{G}_2$ agree on the edges of $E(H_1)\cup E(H_2)$.
    
    Set $q=\lceil\frac{t}{2}\rceil\lfloor\frac{t}{2}\rfloor$, and let $(i_1,j_1),\ldots,(i_q,j_q)$ be an arbitrary ordering of $[\lceil\frac{t}{2}\rceil]\times [\lfloor\frac{t}{2}\rfloor]$. 
    Now for $k=1,\dots,q$, we iteratively construct $\alpha$-tuples $(Y^\ell_k)_{\ell\in [\alpha]}$, of subsets of $S_{i_k}\cup S_{j_k}$ such that $D_k=\Inv(D_{k-1}; (Y^\ell_k)_{\ell\in [\alpha]})$ agrees with $\vec{G}_2$ on 
$E(G\langle S_{i_k}\cup S_{j_k} \rangle)$. This is possible  because $\diam({\mathcal{I}}(G\langle S_{i_k}\cup S_{j_k} \rangle) \leq \alpha$ by assumption.

    We claim that $D_q=\vec{G_2}$. Consider some $uv \in E(G)$ with $u \in S_i$ and $v \in S_{i'}$ for some $i,i' \in [t]$.  
    First suppose that $i=i'$. Let $k \in [q]$ be the largest index such that $i \in \{i_k,j_k\}$. Then, by construction of $Y_k^1,\ldots,Y_k^{\alpha}$, we have that $D_k$ and $\vec{G_2}$ agree on $uv$.
    As $\{u,v\}\cap (Y_{k'}^1\cup \ldots \cup Y_{k'}^{\alpha})=\emptyset$ for all $k' \in \{k+1,\ldots,q\}$, we obtain that $D_q$ and $\vec{G_2}$ agree on $uv$. 
    Now suppose that $i$ and $i'$ are distinct and $i,i' \in [\lceil\frac{t}{2}\rceil]$. 
    Then by definition of $X_1,\ldots,X_p$, we have that $D_0$ and $\vec{G_2}$ agree on $uv$.
    Next, for every $k \in [q]$, at least one of $u,v$ is not in $S_{i_k}\cup S_{j_k}$ and so not in any $Y_k^\ell$ for all $k \in [q]$ and all $\ell \in [\alpha]$. Hence, an immediate induction yields that $D_q$ and $\vec{G_2}$ agree on $uv$.
    Similarly, if $i$ and $i'$ are distinct and $i,i' \in \{\lceil\frac{t}{2}\rceil+1,\ldots,t\}$  we get that $D_q$ and $\vec{G_2}$ agree on $uv$.
    Finally consider the case $i \in [\lceil\frac{t}{2}\rceil]$ and $i' \in \{\lceil\frac{t}{2}\rceil+1,\ldots,t\}$. 
    Then there is some $k \in [q]$ such that $\{i,i'\}=\{i_k,j_k\}$. 
    By construction of $Y_k^1,\ldots,Y_k^{\alpha}$, $D_k$ and $\vec{G_2}$ agree on $uv$. 
    Further, for all $k' \in \{k+1,\ldots,q\}$, there exists some $w \in \{u,v\}$ such that $w \notin S_{i_{k'}}\cup S_{j_{k'}}$. It follows that $\{u,v\}$ is not a subset of $Y_{k'}^\ell$ for all $k' \in \{k+1,\ldots,q\}$ and all $\ell \in [\alpha]$. 
     Hence, $D_q$ and $\vec{G_2}$ agree on $uv$. 
    This shows that $D_q$ and $\vec{G_2}$ agree on all edges, so
    $\Inv(\vec{G_1};(X_1,\ldots,X_p,(Y_{k}^{\ell})_{k \in [q],\ell \in [\alpha]})) = D_q=\vec{G_2}$.
    
    Therefore $\diam ({\mathcal{I}}(G)) \leq p+\alpha q = M (\ID\leq  \lceil\frac{t}{2}\rceil)+\alpha \lceil\frac{t}{2}\rceil\lfloor\frac{t}{2}\rfloor$.
    As $G$ was arbitrary, we get $M (\ID\leq  t)\leq  M (\ID\leq  \lceil\frac{t}{2}\rceil)+\alpha \lceil\frac{t}{2}\rceil\lfloor\frac{t}{2}\rfloor$. 
\end{proof}

We are now ready to prove the following result which improves on Lemma~\ref{lem:H-col}.

\begin{theorem}\label{thm:col-acy-gen}
Let $\alpha$ be a positive integer. 
Then 
$M (\ID\leq  t) \leq \frac{\alpha}{3} t^2+\frac{\alpha}{3}t + \frac{\alpha}{3}$ for all $t\geq 1$.
\end{theorem}

\begin{proof}
    
    
    \medskip
    Let us now prove $M (\ID\leq  t) \leq \frac{\alpha}{3} t^2+\frac{\alpha}{3}t +\frac{\alpha}{3}$ by induction on $t$.
    It is clear from the definitions that $M (\ID\leq  1) \leq M (\ID\leq  2)\leq \alpha$.
    Assume now $t \geq 4$. 
    By Lemma~\ref{lem:recur-ID} and the induction hypothesis, we have

    \begin{align*}
        M (\ID\leq  t) & \leq  M\left(\ID\leq \left\lceil\frac{t}{2}\right\rceil\right)+ \alpha\left\lceil\frac{t}{2}\right\rceil\left\lfloor\frac{t}{2}\right\rfloor \\
                      & \leq  \frac{\alpha}{3}\left\lceil\frac{t}{2}\right\rceil^2+\frac{\alpha}{3}\left\lceil\frac{t}{2}\right\rceil +\frac{\alpha}{3} + \alpha\left\lceil\frac{t}{2}\right\rceil\left\lfloor\frac{t}{2}\right\rfloor. 
    \intertext{If $t$ is even, we obtain }       
    M (\ID\leq  t) 
    & \leq  \frac{\alpha}{3}\frac{t^2}{4}+\frac{\alpha}{3}\frac{t}{2} +\frac{\alpha}{3} + \alpha \frac{t^2}{4} \\
    & \leq  \frac{\alpha}{3}t^2+\frac{\alpha t}{6} +\frac{\alpha}{3} \\
    & \leq  \frac{\alpha}{3}t^2+\frac{\alpha t}{3}+\frac{\alpha}{3}.
    \intertext{If $t$ is odd, we obtain }
    M (\ID\leq  t) & \leq   \frac{\alpha}{3}\left(\frac{t+1}{2}\right)^2+\frac{\alpha}{3}\frac{t+1}{2} +\frac{\alpha}{3} + \alpha \frac{t+1}{2}\frac{t-1}{2}\\
    & =  \frac{\alpha}{3}t^2+\frac{\alpha t}{3} +\frac{\alpha}{3}.
    \end{align*}   

\end{proof}

The following two results improving on Corollary~\ref {cor:chi_s} can immediately be obtained from Theorem~\ref{thm:col-acy-gen}.
\begin{corollary}\label{cor:chi_s-2}\label{cor:L-acyclic-2}
    For every $t \geq 2$, we have $\displaystyle M(\chi_s\leq t) \leq \frac{1}{3}t^2 + \frac{1}{3}t +  \frac{1}{3}$ and
    $M(\chi_a\leq t) \leq \frac{2}{3}t^2 + \frac{2}{3}t + \frac{2}{3}$.
\end{corollary}


We now turn our attention to the task of finding better upper bounds on the inversion diameter in terms of $\chi_o$. 
Using the above-mentioned upper bound of $\chi_o$ in terms of $\chi_a$ due to Kostochka, Sopena, and Zhu~\cite{KSZ97},
with Corollary~\ref{cor:chi_s} or Corollary~\ref{cor:chi_s-2}, we get that $M(\chi_o\leq t) = \bigO \left( t^{6+ 2 \lceil 
\log \log t \rceil}\right)$.
We now establish a much better upper bound as a function of $\chi_o$.

\begin{theorem}\label{thm:diam_L_leq_chi_o_square} $M(\chi_o\leq t) \leq t^2-1$.
\end{theorem}

To prove this theorem, we will need the following lemma.

\begin{lemma}\label{lemma:inversion_n_1_extended}
    Let $\vec{G}_1,\vec{G}_2$ be two orientations of a graph $G$,
    and let $(S_1, \dots , S_k)$ be a proper $k$-colouring of $G$
    such that for every pair of distinct colours $i,j$, the set $E(S_i,S_j)$ of edges between $S_i$ and $S_j$ is either included in $E_{=}$, or included in $E_{\neq}$.
    Then $\dist_{\mathcal{I}(G)}(\vec{G}_1,\vec{G}_2) \leq k-1$.
\end{lemma}

\begin{proof}
    We proceed by induction on $k$.
    For $k=1$, the result is clear as $G$ has no edge.
    Now suppose $k \geq 2$.
    Let $X_{k-1} = S_k \cup \bigcup \{S_i \mid E_{\neq} \cap E(S_i, S_k) \neq \emptyset\}$,
    and let $\vec{G}'_1 = \Inv(\vec{G}_1; X_{k-1}) - S_k$
    and $\vec{G}'_2 = \vec{G}_2 - S_k$.
    
    Observe that $(S_1, \dots , S_{k-1})$ is a proper $(k-1)$-colouring of $G-S_k$,
    and $\vec{G}'_1$ and $\vec{G}'_2$ satisfy the hypothesis of the lemma. 
    Thus, by the induction hypothesis, there exist $X_1, \dots X_{k-2} \subseteq V(G -S_k)$
    such that $\vec{G}'_2 = \Inv(\vec{G}'_1; X_1, \dots ,X_{k-2})$.
    We claim that $\vec{G}_2 = \Inv(\vec{G}_1; X_1, \dots ,X_{k-1})$.
    Indeed, for every edge $uv \in E(G)$, if $u\in S_k$ or $v\in S_k$,
    then $\{u,v\} \subseteq X_{k-1}$ if and only if $uv \in E_{\neq}$ (by assumption on the colouring).
    Otherwise $u,v \in V(\vec{G}'_1)=V(\vec{G}'_2)$, and $uv$ is inverted
    by $X_1, \dots , X_{k-2}$ if and only if it is oriented differently
    in $\Inv(\vec{G}_1; X_{k-1})$ and $\vec{G}_2$.
    This proves the lemma.
\end{proof}

We are now ready to give the main proof of Theorem~\ref{thm:diam_L_leq_chi_o_square}.

\begin{proof}[Proof of Theorem~\ref{thm:diam_L_leq_chi_o_square}]
    Let $\vec{G}_1,\vec{G}_2$ be two orientations of a graph $G$ with oriented chromatic number $\chi_o(G)=t$.
    For every $\ell\in \{1,2\}$, $\vec{G}_i$ has an oriented colouring $(S_1^\ell, \dots , S_t^\ell)$.
    For every pair $(i,j) \in [t]\times [t]$, let 
    $S_{i,j} = S_i^1\cap S_j^2$. 
    Observe that, for any two distincts pairs $(i,j)$ and $(i',j')$, the set $E(S_{i,j}, S_{i',j'})$  is either included in $E_=$, or included in $E_{\neq}$.
    Hence, by Lemma~\ref{lemma:inversion_n_1_extended}, we obtain
    $\dist_{\mathcal{I}(G)}(\vec{G}_1,\vec{G}_2) \leq t^2-1$.
\end{proof}

\section{Better bounds in terms of the maximum degree}\label{sec:Delta}

The fact that the inversion diameter of a graph is upper bounded by a function of its maximum degree follows from Corollary~\ref{cor:dqewdrqwde} as well as the results established in Section~\ref{sec:carac}.
By Theorem~\ref{thm:multipartite}, a complete graph $K$ satisfies $\diam(\mathcal{I}(K)) = \Delta(K)$, so 
$M(\Delta \leq k) \geq k$.
We conjecture that equality holds $M(\Delta \leq k) = k$.

\begin{conjecture}\label{conj:Delta}
    For every graph $G$, we have $\diam(\mathcal{I}(G)) \leq \Delta(G)$.
\end{conjecture}

In this section, we give some evidences for this conjecture to be true.
First, in Subsection~\ref{sec:mdarb}, we show that $M(\Delta \leq k)\leq 2k-1$ for every positive integer $k$, and that $\diam(\mathcal{I}(G)) \leq \Delta(G) + o(\Delta(G))$ for every bipartite graph $G$.
In Sections~\ref{sec:md2} and~\ref{sec:md3}, we prove $M(\Delta \leq 2) = 2$ and
$M(\Delta \leq 3) \leq 4$.

\subsection{Bounds for arbitrary maximum degree}\label{sec:mdarb}
In this subsection, we prove that $M(\Delta \leq k)\leq 2k-1$ for every positive integer $k$, which can be restated as follows. 

\begin{theorem}\label{thm:L_2Delta-1}
    For every graph $G$ with at least one edge, $\diam(\mathcal{I}(G)) \leq 2\Delta(G) - 1$.
\end{theorem}

We need the following lemma.

\begin{lemma}\label{lem:technical_bound_for_2Delta-1}
    For every positive $\Delta$ with  $\Delta \geq 2$, for every integer $d$ with $0\leq d\leq \Delta$, 
    \[
        2^{2\Delta-1 -d} > 1 + (\Delta - d)(2^{\Delta-1}-1).
    \]
\end{lemma}
\begin{proof}
    If $\Delta=d$, we have $2^{2\Delta-1 -d}=2^{\Delta-1 }\geq 2>1=1 + (\Delta - d)(2^{\Delta-1}-1)$. Otherwise, as $k<2^k$ holds for every nonnegative integer $k$, we have 
    \begin{align*}
        1 + (\Delta - d)(2^{\Delta-1}-1)&\leq (\Delta - d)2^{\Delta-1}\\
        &<2^{\Delta - d}2^{\Delta-1}\\
        &=2^{2\Delta-1 -d}.\qedhere
     \end{align*}
\end{proof}

\begin{proof}[Proof of Theorem~\ref{thm:L_2Delta-1}]\label{thm:Delta}
    If $\Delta(G)=1$, then $\diam(\mathcal{I}(G)) = 1$.
    Otherwise, let $\Delta=\Delta(G)$ and let $<$ be an arbitrary ordering of $V(G)$. Now consider some $u \in V(G)$ and let $d = |N_>(u)|$. If $N_{>}(u) \neq \emptyset$, then we have
    \begin{align*}
        |N_<(u)| &+ 
        \log \left(|\{X \subseteq V(G) \mid \exists~v \in N_>(u), X \subseteq  N_<(v) \}|\right) \\
        &\leq d + \log(1+(\Delta-d)(2^{\Delta-1}-1)) \\
        &<2\Delta - 1~ \mbox{by Lemma~\ref{lem:technical_bound_for_2Delta-1}}.
    \end{align*}
    If $N_{>}(u) = \emptyset$, then $|N_<(u)| \leq \Delta - 1 \leq 2\Delta-1$.
    Hence $G$ is strongly $(2\Delta - 1)$-degenerate, so $\diam(\mathcal{I}(G)) \leq 2\Delta - 1$ by Theorem~\ref{thm:generic_greedy_argument}.
\end{proof}

For bipartite graphs, Theorem~\ref{thm:L_2Delta-1} can be improved as follows.

\begin{theorem}\label{thm:L_bipartite}
$\diam(\mathcal{I}(G)) \leq \Delta + \lceil\log\Delta\rceil - 1$ for every bipartite graph $G$ of maximum degree $\Delta \geq 2$.
\end{theorem}

\begin{proof}
    Let $t=\Delta + \lceil\log\Delta\rceil - 1$.
    Let $G$ be a bipartite graph of maximum degree $\Delta$ and let $(A,B)$ be a bipartition of $G$.
    Let $<$ be an ordering of $V(G)$ such that $a<b$ for every $a \in A$ and $b \in B$.
    For every $a \in A$ we have $N_<(a) = \emptyset$ and $|\{X \subseteq V(G) \mid \exists~v \in N_>(a), X \subseteq N_{<a}(v)\}| \leq 1+\Delta(2^{\Delta-1}-1)<\Delta 2^{\Delta-1} \leq 2^t$.
    For every $b \in B$, $|N_<(b)| \leq \Delta \leq t$ and $|\{X \subseteq V(G) \mid \exists~v \in N_>(b), X \subseteq N_<(v)\}|=0$.
    Hence $<$ is a $t$-strong ordering of $G$. Thus, by Theorem~\ref{thm:generic_greedy_argument}, we have $\diam(\mathcal{I}(G)) \leq t$.
\end{proof}

\subsection{Maximum degree at most \texorpdfstring{$2$}{2}}
\label{sec:md2}

We here prove the following result for graphs of maximum degree at most 2.

\begin{theorem}\label{thm:L-cycle}
    Let $G$ be a graph.
    If $\Delta(G) \leq 2$, then $\diam ({\mathcal{I}}(G)) \leq 2$.  
\end{theorem}

\begin{proof}
    By possibly adding some edges, we can assume that every connected component of $G$ is a cycle or a single edge.
    Moreover, as $\diam(\mathcal{I}(G))$ is the maximum of $ \diam(\mathcal{I}(C))$ over all components $C$ of $G$, and since $\diam(\mathcal{I}(K_2)) = 1$, it is enough to prove the statement for connected components $C$ of $G$ which are a cycle.
    
    Let $\vec{C}_1$ and $\vec{C}_2$ be two orientations of a cycle $C$.
    First assume that $E_=$ contains two consecutive edges.
    Let $x$ be the common vertex of those two edges.
    By Theorem~\ref{thm:L-arbre}, there are two sets $X_1$ and $X_2$ of $V(C-x)$ such that
    $\Inv(\vec{C}_1-x ; X_1, X_2)= \vec{C}_2-x$.
    Then $\Inv(\vec{C}_1 ; X_1, X_2)= \vec{C}_2$.
    
    Now assume that there are no two consecutive edges in $E_{=}$. 
    If $E_{=}=\emptyset$, then inverting $V(C)$ transforms 
    $\vec{C}_1$ into $\vec{C}_2$.
    Henceforth, we may assume that $E_{=}\neq \emptyset$.
    Let $P_1, P_2 , \dots , P_t$ be the connected components of $C\setminus E_{=}$.
    If $t$ is even, let $X_1=\bigcup_{i=1}^{t/2} V(P_{2i-1})$ and $X_2=\bigcup_{i=1}^{t/2} V(P_{2i})$; we have $\Inv(\vec{C}_1 ; X_1, X_2)= \vec{C}_2$.
    If $t$ is odd, set $s=\frac{t-1}{2}$ and let $X_1= \bigcup_{i=0}^{s} V(P_{2i+1})$ and $X_2= \{x,y\} \cup \bigcup_{i=1}^{s} V(P_{2i})$ where $xy$ is the edge between $P_1$ and $P_t$; we have $\Inv(\vec{C}_1 ; X_1, X_2)= \vec{C}_2$.
\end{proof}

\subsection{Subcubic graphs}\label{sec:md3}

For subcubic graphs, we can slightly improve on Theorem~\ref{thm:L_2Delta-1}.
First we need some preliminaries.


Let $G$ be a subcubic multigraph, $\pi \colon E(G) \to \mathbb{F}_2$ a mapping, and let $\sigma = (v_1, \dots, v_n)$ be an ordering of $V(G)$.
We say that a vertex $v_i$ is {\bf bad} for $\sigma$ if there exist $j_1, j_2, j_3$ such that 
$j_1 \leq j_2 < i < j_3$, $v_{j_1}v_i,v_{j_2}v_i$, and $v_{j_3}v_i$ are distinct edges of $E(G)$, and $\pi(v_{j_1}v_i) = \pi(v_{j_2}v_i)=0$ and $\pi(v_iv_{j_3})=1$. 
The ordering $\sigma$ is {\bf good} if it has no bad vertex.
The following result is the main ingredient in the proof of Theorem~\ref{thm:L-cubic}.

\begin{lemma}\label{lem:good-ordering}
    Let $G$ be a subcubic multigraph and $\pi \colon E(G) \to \mathbb{F}_2$ a mapping.
    Then there is a good ordering of $V(G)$.
\end{lemma}

\begin{proof}
Suppose otherwise and let $(G,\pi)$ be a minimum counterexample to the statement. 
We say that an edge $e=uv \in E(G)$ is {\it critical} if $\pi(e)=1, d_G(u)=3$, $d_G(v)=3$, $E(G)\setminus e$ does not contain an edge which is parallel to $e$, and $\pi(e')=0$ for every $e' \in E(G)\setminus e$ which is incident to at least one of $u$ and $v$.
\begin{claim}\label{clm:crit}
    No edge in $E(G)$ is critical.
\end{claim}
\begin{subproof}
    Suppose otherwise, so $E(G)$ contains a critical edge $e_0=uv$. 
    Let $ut_1$ and $ut_2$, and $vw_1$ and $vw_2$ be the edges distinct from $e$ incident to $u$ and $v$, respectively. 
    Observe that $\{u,v\}\cap \{t_1,t_2,w_1,w_2\}=\emptyset$ by assumption. 
    In the following, we distinguish several cases on the structure of these neighbourhoods, and obtain contradictions which eventually prove the claim. 
    By symmetry, we may suppose that one of the following four cases occurs.
    \begin{cas}
        $t_1=t_2$ and $w_1=w_2$.
    \end{cas}
    Let $G'$ be obtained from $G-\{u,v\}$ by adding two edges $e_1,e_2$ linking $t_1$ and $w_1$. 
    We further define $\pi'\colon E(G') \to \mathbb{F}_2$ by $\pi'(e_1)=\pi'(e_2)=0$ and $\pi'(e)=\pi(e)$ for all $e \in E(G)\cap E(G')$. 
    As $G'$ is subcubic and smaller than $G$, there exists a good ordering $\sigma'=(v_1,\ldots,v_{n-2})$ of $V(G')$ with respect to $\pi'$. 
    Let $i,j \in [n-2]$ such that $v_i=t_1$ and $v_j=w_1$. Since $G$ is subcubic, we have $i \neq j$ and by symmetry, we may suppose that $i<j$. 
    Now consider the ordering $\sigma=(v_1,\ldots,v_i,v,u,v_{i+1},\ldots,v_{n-2})$. We obtain that none of $t_1,u$, and $v$ are bad in $\sigma$. 
    Furthermore, $w_1$ is not bad in $\sigma$ as $\sigma'$ is a good ordering of $V(G')$ with respect to $\pi'$. 
    Finally, no vertex in $V(G)\setminus\{u,v,t_1,w_1\}$ is bad in $\sigma$ as $\sigma'$ is a good ordering of $V(G')$ with respect to $\pi'$. 
    It follows that $\sigma$ is a good ordering of $V(G)$ with respect to $\pi$, a contradiction to the choice of $(G,\pi)$.
    \begin{cas}
        $t_2=w_1=w_2$.
    \end{cas}
    Let $G'=G-\{u,v,t_2\}$ and $\pi'=\pi\vert_{E(G')}$. As $G'$ is subcubic and smaller than $G$, there exists a good ordering $\sigma'=(v_1,\ldots,v_{n-3})$ of $V(G')$ with respect to $\pi'$. Now consider the ordering $\sigma=(v_1,\ldots,v_{n-3},u,v,t_2)$.  We obtain that none of $t_1,t_2,u$, and $v$ are bad in $\sigma$. Further, no vertex in $V(G)\setminus\{u,v,t_1,t_2\}$ is bad in $\sigma$ as $\sigma'$ is a good ordering of $V(G')$ with respect to $\pi'$. It follows that $\sigma$ is a good ordering of $V(G)$ with respect to $\pi$, a contradiction to the choice of $(G,\pi)$.
    \begin{cas}
        $w_1=w_2$ and $t_1,t_2$, and $w_1$ are pairwise distinct.
    \end{cas}
    Let $G'$ be obtained from $G-\{u,v\}$ by adding two edges: $e_1$ with endvertices $t_1,w_1$, and $e_2$ with endvertices $t_2,w_1$. We further define $\pi' \colon E(G') \to \mathbb{F}_2$ by $\pi'(e_1)=\pi'(e_2)=0$ and $\pi'(e)=\pi(e)$ for all $e \in E(G)\cap E(G')$. As $G'$ is subcubic and smaller than $G$, there exists a good ordering $\sigma'=(v_1,\ldots,v_{n-2})$ of $V(G')$ with respect to $\pi'$.  Let $i,j,k \in [n-2]$ be such that $v_i=t_1$, $v_j = t_2$ and $v_k = w_1$.
    Note that $i,j,$ and $k$ are pairwise distinct. By symmetry, we may suppose that $i<j$.
    \begin{itemize}
        \item[(a)] If $i<k<j$, then let $\sigma = (v_1, \dots, v_{k-1}, u, w_1, v, v_{k+1}, \dots, v_{n-2})$.
        \item[(b)] If $k<i<j$, then let $\sigma = (v_1, \dots, v_{k-1}, u, w_1, v, v_{k+1}, \dots, v_{n-2})$.
        \item[(c)] If $i<j<k$, then let $\sigma = (v_1, \dots, v_{k-1}, v, u, w_1, v_{k+1}, \dots, v_{n-2})$.
    \end{itemize}
    All of these cases are illustrated in Figure~\ref{fig:good_ordering}.
    Similar arguments as above show that $\sigma$ is a good ordering of $G$ with respect to $\pi$. This contradicts the choice of $(G,\pi)$.

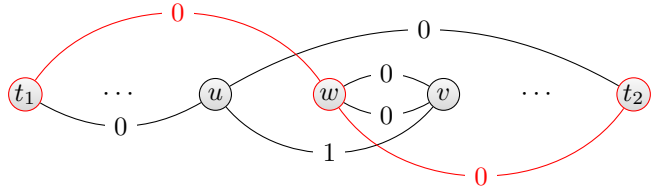
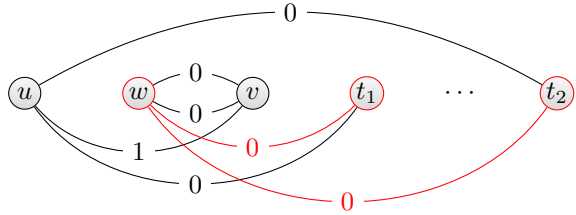
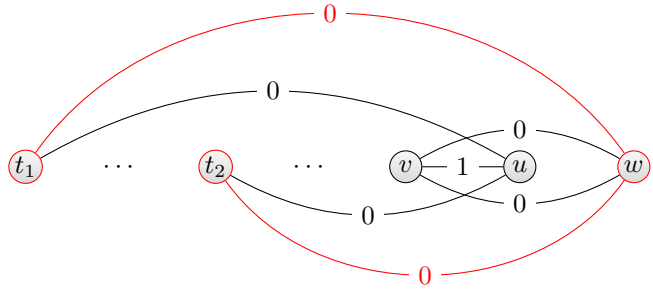
\begin{figure}[!hbt]
    \centering
    \begin{subfigure}[ht]{0.9\textwidth}
        \centering
        \begin{tikzpicture}
            \tikzstyle{vertex}=[circle,draw, top color=gray!5, 
            bottom color=gray!30, minimum size=12pt, scale=1, inner sep=0.5pt]
            \tikzstyle{vertexR}=[circle,draw=red, top color=gray!5, 
            bottom color=gray!30, minimum size=12pt, scale=1, inner sep=0.5pt]
            \node[vertexR] (t1) at (0-1,0) {$t_1$};
            \node[vertex] (u) at (1.5,0) {$u$};
            \node[vertexR] (w) at (3,0) {$w$};
            \node[vertex] (v) at (4.5,0) {$v$};
            \node[vertexR] (t2) at (6+1,0) {$t_2$};

            \node at (0.25,0) {$\cdots$};
            \node at (5.75,0) {$\cdots$};

            \draw[bend right] (t1) to node[midway, fill=white] {$0$} (u);
            \draw[bend left] (u) to node[midway, fill=white] {$0$} (t2);
            \draw[bend right=50] (u) to node[midway, fill=white] {$1$} (v);
            \draw[bend left] (v) to node[midway, fill=white] {$0$} (w);
            \draw[bend right] (v) to node[midway, fill=white] {$0$} (w);

            \draw[red, bend left=55] (t1) to node[midway, fill=white] {$0$} (w);
            \draw[red, bend left=55] (t2) to node[midway, fill=white] {$0$} (w);
        \end{tikzpicture}
        \caption{Case 3 (a) : $i<k<j$.}
    \end{subfigure}
    
    \begin{subfigure}[ht]{0.9\textwidth}
        \centering
        \begin{tikzpicture}
            \tikzstyle{vertex}=[circle,draw, top color=gray!5, 
            bottom color=gray!30, minimum size=12pt, scale=1, inner sep=0.5pt]
            \tikzstyle{vertexR}=[circle,draw=red, top color=gray!5, 
            bottom color=gray!30, minimum size=12pt, scale=1, inner sep=0.5pt]
            \node[vertexR] (t1) at (4.5,0) {$t_1$};
            \node[vertex] (u) at (0,0) {$u$};
            \node[vertexR] (w) at (1.5,0) {$w$};
            \node[vertex] (v) at (3,0) {$v$};
            \node[vertexR] (t2) at (6+1,0) {$t_2$};

            \node at (5.75,0) {$\cdots$};

            \draw[bend left=55] (t1) to node[midway, fill=white] {$0$} (u);
            \draw[bend left] (u) to node[midway, fill=white] {$0$} (t2);
            \draw[bend right=50] (u) to node[midway, fill=white] {$1$} (v);
            \draw[bend left] (v) to node[midway, fill=white] {$0$} (w);
            \draw[bend right] (v) to node[midway, fill=white] {$0$} (w);

            \draw[red, bend left=45] (t1) to node[midway, fill=white] {$0$} (w);
            \draw[red, bend left=55] (t2) to node[midway, fill=white] {$0$} (w);
        \end{tikzpicture}
        \caption{Case 3 (b) :  $k<i<j$.}
    \end{subfigure}
    
    \begin{subfigure}[ht]{0.9\textwidth}
        \centering
        \begin{tikzpicture}
            \tikzstyle{vertex}=[circle,draw, top color=gray!5, 
            bottom color=gray!30, minimum size=12pt, scale=1, inner sep=0.5pt]
            \tikzstyle{vertexR}=[circle,draw=red, top color=gray!5, 
            bottom color=gray!30, minimum size=12pt, scale=1, inner sep=0.5pt]
            \node[vertexR] (t1) at (0-2,0) {$t_1$};
            \node[vertex] (u) at (4.5,0) {$u$};
            \node[vertexR] (w) at (6,0) {$w$};
            \node[vertex] (v) at (3,0) {$v$};
            \node[vertexR] (t2) at (1.5-1,0) {$t_2$};

            \node at (-0.75,0) {$\cdots$};
            \node at (1.75,0) {$\cdots$};

            \draw[bend left] (t1) to node[midway, fill=white] {$0$} (u);
            \draw[bend left] (u) to node[midway, fill=white] {$0$} (t2);
            \draw (u) to node[midway, fill=white] {$1$} (v);
            \draw[bend left] (v) to node[midway, fill=white] {$0$} (w);
            \draw[bend right] (v) to node[midway, fill=white] {$0$} (w);

            \draw[red, bend left=55] (t1) to node[midway, fill=white] {$0$} (w);
            \draw[red, bend right=55] (t2) to node[midway, fill=white] {$0$} (w);
        \end{tikzpicture}
        \caption{Case 3 (c) : $i<j<k$.}
    \end{subfigure}
    \caption{Some cases in the proof of Lemma~\ref{lem:good-ordering}. The vertices are placed from left to right according to $\sigma$.
    The edges in $E(G') \setminus E(G)$ and the vertices in $V(G')$ are depicted in red.}
    \label{fig:good_ordering}
\end{figure}
    \begin{figure}[!hbtp]
    \centering
    \begin{tikzpicture}
        \tikzstyle{vertex}=[circle,draw, top color=gray!5, 
        bottom color=gray!30, minimum size=12pt, scale=1, inner sep=0.5pt]
        \tikzstyle{vertexR}=[circle,draw=red, top color=gray!5, 
        bottom color=gray!30, minimum size=12pt, scale=1, inner sep=0.5pt]
        \node[vertexR] (t1) at (-3-1,0) {$t_1$};
        \node[vertex] (u) at (-1.5,0) {$u$};
        \node[vertexR] (w1) at (0,-1.5) {$w_1$};
        \node[vertexR] (w2) at (1+3,0) {$w_2$};
        \node[vertex] (v) at (1+1.5,0) {$v$};
        \node[vertexR] (t2) at (0,0) {$t_2$};

        \node at (-2.75,0) {$\cdots$};
        \node at (1.25,0) {$\cdots$};

        \draw[bend left=55] (t1) to node[midway, fill=white] {$0$} (u);
        \draw (u) to node[midway, fill=white] {$0$} (t2);
        \draw[bend left=55] (u) to node[midway, fill=white] {$1$} (v);
        \draw (v) to node[midway, fill=white] {$0$} (w1);
        \draw (v) to node[midway, fill=white] {$0$} (w2);

        \draw[red, bend right] (w1) to node[midway, fill=white] {$0$} (w2);
        \draw[red, bend right=55] (t1) to node[midway, fill=white] {$0$} (t2);
    \end{tikzpicture}
    \caption{The case $i<j,k<l,j<l$ in the proof of Lemma~\ref{lem:good-ordering}, with vertices placed from left to right according to $\sigma$.
    Since the relative position of $w_1$ with $t_1$ and $t_2$ is unknown, $w_1$ is placed below.
    The edges in $E(G') \setminus E(G)$ and the vertices in $V(G')$ are depicted in red.}
    \label{fig:good_ordering_general_case}
\end{figure}
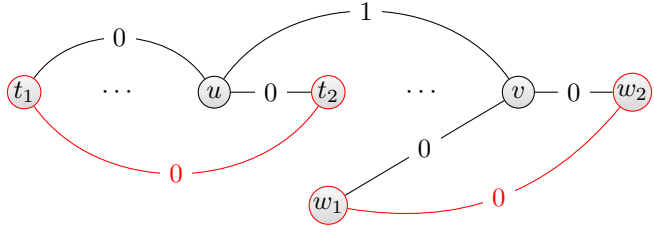
    
    \begin{cas}
        $t_1\neq t_2$ and $w_1\neq w_2$.
        \end{cas}
    Let $G'$ be obtained from $G-\{u,v\}$ by adding two edges: $e_1$ with endvertices $t_1,t_2$, and $e_2$ with endvertices $w_1,w_2$. 
    We further define $\pi' \colon E(G') \to \mathbb{F}_2$ by $\pi'(e_1)=\pi'(e_2)=0$ and $\pi'(e)=\pi(e)$ for all $e \in E(G)\cap E(G')$. 
    As $G'$ is subcubic and smaller than $G$, there exists a good ordering $\sigma'=(v_1,\ldots,v_{n-2})$ of $V(G')$ with respect to $\pi'$. 
    Let $i,j,k,\ell \in [n-2]$ be such that $v_i=t_1$, $v_j = t_2,v_k=w_1$, and $v_\ell = w_2$. Without loss of generality, we may suppose that $i<j, k < \ell$, and $j \leq \ell$.  
    If $j<\ell$, then let 
    \[
        \sigma = (v_1, \dots ,  v_{j-1}, u, t_2, v_{j+1}, \dots, v_{\ell-1}, v, w_2, v_{\ell +1}, \dots, v_{n-2}).
    \]
    This case is illustrated in Figure~\ref{fig:good_ordering_general_case}.
    If $j=\ell$, then observe that $t_2=w_2$. Let
    \[
        \sigma = (v_1, \dots ,  v_{j-1}, u, v, w_2, v_{j +1}, \dots, v_{n-2}).
    \]
    In either case, we obtain that none of $t_1,u,w_1$, and $v$ is bad in $\sigma$. 
    Furthermore, $t_2$ and $w_2$ are not bad in $\sigma$ as $\sigma'$ is a good ordering of $V(G')$ with respect to $\pi'$. 
    Finally, no vertex in $V(G)\setminus\{u,v,t_1,w_1\}$ is bad in $\sigma$ as $\sigma'$ is a good ordering of $V(G')$ with respect to $\pi'$. 
    It follows that $\sigma$ is a good ordering of $V(G)$ with respect to $\pi$, a contradiction to the choice of $(G,\pi)$.
    
    This finishes the proof of the claim.
\end{subproof}

Let $V_0$ be the set of vertices $v \in V(G)$ with $d_G(v)=3$ that are incident to exactly one edge in $\pi^{-1}(1)$. 
Now consider  an ordering $\sigma$ of $V(G)$ in which all vertices in $V_0$ are preceded by all vertices in $V(G)\setminus V_0$. 
Clearly, no vertex in $V(G)\setminus V_0$ is bad in $\sigma$. Now consider some vertex $v \in V_0$ and let $w$ be its unique neighbour such that $\pi(vw)=1$. 
By Claim~\ref{clm:crit}, we obtain that $w \in V(G)\setminus V_0$. By construction, it follows that $w<v$ with respect to $\sigma$. 
Hence $v$ is not bad with respect to $\sigma$, and it follows that $\sigma$ is a good ordering of $G$ with respect to $\pi$. 
This contradicts the choice of $(G,\pi)$.
\end{proof}

We now give the main result of this section. The approach is very similar to the one of Theorem~\ref{thm:generic_greedy_argument}. We use the ordering whose existence is guaranteed by Lemma~\ref{lem:good-ordering} and make some finer case distinctions.
\begin{theorem}\label{thm:L-cubic}
Let $G$ be a graph.
If $\Delta(G)\leq 3$, then $\diam(\mathcal{I}(G)) \leq 4$.
\end{theorem}

\begin{proof}
    Let $G$ be a graph of maximum degree at most $3$.
    Consider a  function $\pi \colon  E(G) \to \mathbb{F}_2$.
    We show that there is a family $(\mathbf{u})_{u \in V(G)}$ of vectors in $\mathbb{F}_2^4$ such that $\pi(uv) = \mathbf{u} \cdot \mathbf{v}$ for every edge $uv \in E(G)$.
    If there is a vertex $u$ incident only to edges $e$ with $\pi(e)=0$,
    then we apply induction on $G-u$ and set $\mathbf{u}=\mathbf{0}$.
    Now we assume that every vertex of $G$ is incident to at least one edge $e$ with $\pi(e)=1$.

   By Lemma~\ref{lem:good-ordering}, there is a good ordering $\sigma = (v_1, \dots, v_n)$ of $G$ with respect to $\pi$. 
    We iteratively construct $\mathbf{v}_1, \dots, \mathbf{v}_n \in \mathbb{F}_2^4$ with the following property for every $i \in [n]$:
    \begin{enumerate}
        \item\label{item:cubic_lin_indep}  for every $k>i$, $(\mathbf{v}_j)_{j \leq i, v_j \in N(v_k)}$ are linearly independent, and 
        \item\label{item:cubic_correct_orientation} for every edge $v_j v_{j'}$ with $j,j' \leq i$, $\mathbf{v}_j \cdot \mathbf{v}_{j'} = \pi(v_j v_{j'})$. 
    \end{enumerate}
    Initially, it suffices to take $\mathbf{v}_1 \neq \mathbf{0}$.
    Now suppose that $\mathbf{v}_1, \dots, \mathbf{v}_i$ satisfy this invariant for some $i \in [n-1]$.
    We denote by $d$ the number of neighbours of $v_{i+1}$ in $\{v_1, \dots, v_i\}$.
    Condition~\ref{item:cubic_lin_indep} forbids for each neighbour $v_k$ of $v_{i+1}$ with $k>i+1$ a vector space of dimension at most $3-1=2$, and so in total,
    their union $U_1$ has size at most $1+(3-d)(2^{2}-1)$ and contains $\mathbf{0}$.
    Condition~\ref{item:cubic_correct_orientation} allows at least $2^{4-d}$ values of $\mathbf{v}_{i+1}$. 
    Indeed, as $(\mathbf{v}_j)_{v_j \in N(v_{i+1}), j \leq i}$ are linearly independent, the affine space $U_2$
    of the solutions of the system of equations$, \mathbf{v}_j \cdot \mathbf{v}_{i+1} = \pi(v_j v_{i+1})$, for $v_j \in N(v_{i+1})$ with $j \leq i$, has dimension at least $4-d$.
    It is enough to show that there is a vector in $U_2 \setminus U_1$.
    If $d = 0$, we have $|U_2| = 2^4 > 1 + 3 \times 3 \geq |U_1|$. If $d = 1$, we obtain $|U_2| = 2^3 > 1 + 2 \times 3 \geq |U_1|$. If $d = 3$, then $|U_2| = 2 > 1  \geq |U_1|$. As $G$ is subcubic, it remains to consider the case $d=2$. If $d_G(v_{i+1})=2$, we obtain $|U_1| =0 < 2^2 = |U_2|$. We may hence suppose that there exists a unique neighbour $v_k$ of $v_{i+1}$ with $k>i+1$. If $v_k$ has a higher-indexed neighbour, then $|U_1| \leq 2^1 < 2^2 = |U_2|$. We may hence suppose that $v_k$ has no higher-indexed neighbour. As $v_{i+1}$ is not bad in $\sigma$ and  $v_{i+1}$ is incident to at least one edge in $\pi^{-1}(1)$, $v_{i+1}$ has a neighbour $v_j$ with $j<i+1$ such that $\pi(v_j v_{i+1})=1$. This implies that $\mathbf{0} \not\in U_2$,
    and so $|U_2 \setminus U_1| = |U_2 \setminus (U_1 \setminus \{\mathbf{0}\})| \geq 2^2 - 3 >0$.

    In each case, we found $\mathbf{v}_{i+1}$ such that $\mathbf{v}_1, \dots, \mathbf{v}_{i+1}$ satisfies the invariant.
    At the end of the process, we have $\mathbf{v}_1, \dots, \mathbf{v}_n$ such that for every edge $v_iv_j$, $\mathbf{v}_i \cdot \mathbf{v}_j = \pi(v_i v_j)$.
    Using Observation~\ref{obs:characterization_with_vectors}, this proves the theorem.
\end{proof}


It follows from Theorem~\ref{thm:multipartite} applied to $K_4$ that the bound in Theorem~\ref{thm:L-cubic} cannot be improved further than 3.
Quite surprisingly, the following result shows that $K_4$ is not a particular example with this property. Indeed, the inversion diameter of any cubic graph is at least $3$.

\begin{theorem}\label{thm:every_cubic_at_least_3}
  Let $G$ be a graph. 
    If $\delta(G)\geq 3$, then $\diam(\mathcal{I}(G)) \geq 3$.
\end{theorem}

Hence, a positive answer to the aforementioned problem would show that the inversion diameter of all cubic graphs is exactly $3$.

Theorem~\ref{thm:every_cubic_at_least_3} easily follows from the following two lemmas, the latter of which will be reused in Section~\ref{sec:minor_closed}.

\begin{lemma}\label{lem:cubic_even_cycle}
    Let $G$ be a graph. 
    If $\delta(G)\geq 3$, then $G$ has an even cycle.
\end{lemma}

\begin{proof}
    Let $P$ be a path of maximal length in $G$ and let $y$ be an endvertex of $P$, and let $z$ be its predecessor along $P$.  As $P$ is maximal, $N(y) \subseteq V(P)$. For any two vertices $u,v$ of $P$, we denote by $P[u,v]$ the subpath of $P$ with endvertices $u$ and $v$.
    Let $x_1,x_2 \in V(P)$ be such that $N(y) = \{x_1,x_2,z\}$ with $x_2$ closer to $y$ than $x_1$ in $P$.
   Consider the three cycles $C_1 = P[x_1,x_2] \cup x_2 yx_1$, $C_2 = P[x_2,y] \cup yx_2$ and $C_3 =  P[x_1,y] \cup yx_1$.
    Observe that $|V(C_1)| + |V(C_2)| = |V(C_3)| + 2$, thus at least one of $|V(C_1)|$, $|V(C_2)|$ or $|V(C_3)|$ is even, and so $G$ has an even cycle.
\end{proof}

\begin{lemma}\label{lem:aumoins3}
    Let $G$ be a graph that contains an even cycle $C$ with $d_G(v)\geq 3$ for all $v \in V(C)$. Then $\diam(\mathcal{I}(G)) \geq 3$.
\end{lemma}
\begin{proof}
    Let $C=v_1v_2 \cdots v_gv_1 $.
    Let $\vec{G}_1, \vec{G}_2$ be two orientations of $G$ that agree on $E(C)\setminus\{v_g v_1\}$, and disagree on all other edges in $E(G)$.
    Suppose, for the sake of a contradiction, that there is an inversion sequence $(X_1,X_2)$ that transforms $\vec{G}_1$ into $\vec{G}_2$.
    As $d_G(v)\geq 3$ for all $v \in V(C)$, every vertex in $V(C)$ is incident to an edge of $E_{\neq}$. 
    Hence $X_1 \cup X_2$ covers $V(C)$.
    
    If $v_i \in X_1 \cap X_2$ for some $i\in \{1,\ldots,g-1\}$, then $v_{i+1} \in X_1 \cap X_2$  
    as $v_i v_{i+1} \in E_{=}$. 
    Similarly, if $v_i \in X_1 \cap X_2$ for some $i\in \{2,\ldots,g\}$, then $v_{i-1} \in X_1 \cap X_2$. 
    Repeatedly applying these arguments, if there is some $i \in \{1,\ldots,g-1\}$ with $v_i \in X_1 \cap X_2$, 
    then $v_1,v_2, \dots ,v_g \in X_1 \cap X_2$. 
    However, it then follows that $v_g v_1$ is not reversed, a contradiction.
    Now suppose that no vertex in $V(C)$ belongs to $X_1 \cap X_2$.
    
    Without loss of generality, suppose that $v_1 \in X_1$.
    For every $i \in [g-1]$ and $j\in [2]$, we have $v_i \in X_j$ if and only if $v_{i+1} \in X_{3-j}$ as $v_i v_{i+1}\in E_{=}$.
    Thus, by induction, $v_i \in X_1$ if $i$ odd and $v_i \in X_2$ if $i$ even.
    In particular, $v_g \in X_2$, but then $v_g v_1$ is not reversed, a contradiction.
\end{proof}

We are now ready to finish the proof of Theorem~\ref{thm:every_cubic_at_least_3}.

\begin{proof}[Proof of Theorem~\ref{thm:every_cubic_at_least_3}]
    By Lemma~\ref{lem:cubic_even_cycle}, $G$ has an even cycle $C$. 
    As $G$ is cubic, we have $d_G(v)\geq 3$ for all $v \in V(C)$. It now follows from Lemma~\ref{lem:aumoins3} that $\diam(\mathcal{I}(G)) \geq 3$.
\end{proof}

\section{Bounds in terms of treewidth}\label{sec:tw}

The treewidth of a graph $G$ is denoted by $\tw(G)$. A formal definition of treewidth can for example be found in~\cite{diestel}.
In this section, we consider $M(\tw\leq t)$, which is the maximum inversion diameter of a graph with treewidth at most $t$. 
The graphs with treewidth $1$ are the forests, so $M(\tw\leq 1)=2$ by Corollary~\ref{cor:forest}. 
We shall prove $t+2 \leq M(\tw\leq t) \leq 2t$ for all $t\geq 2$. We start by proving the upper bound $M(\tw\leq t) \leq 2t$.


\begin{theorem}\label{thm:Mtw2t}
    $\diam(\mathcal{I}(G)) \leq 2 \tw(G)$ for every  graph $G$. 
\end{theorem}

\begin{proof}
    Let $G$ be a graph of treewidth $t$.
    Let $(T,(B_x)_{x \in V(T)})$ be a tree-decomposition of $G$ of width $t$.
    Since adding edges does not decrease the inversion diameter, we can assume that for every $x \in V(T)$, $B_x$ induces a complete graph in $G$.
    Root $T$ on an arbitrary vertex, and consider an ordering $<$ of $V(G)$ satisfying the following condition: 
    for every pair $u,v$ of vertices, if the root of the subtree of $T$ induced by $\{y \in V(T) \mid u \in B_y\}$ 
    is an ancestor of the root of the subtree induced by $\{y \in V(T) \mid v \in B_y\}$, then $u<v$.
    One can easily check that such an ordering $<$ exists and that for every vertex $u$,
    if $x$ is the root of the subtree induced by $\{y \in V(T) \mid u \in B_y\}$, then
    \begin{enumerate}
        \item $N_{<}(u) \subseteq B_x \setminus \{u\}$, and
        \item for every $v \in N_{>}(u)$, $N_{<u}(v) \subseteq B_x \setminus \{u\}$.
    \end{enumerate}
    First suppose that $N_{>}(u) \neq \emptyset$.
    We show that if $|B_x|=t+1$, then there is no vertex $v \in N_>(u)$ such that $B_x \setminus \{u\}\subseteq N(v)$. 
    Otherwise, let $x' \in V(T)$ be the node at which the subtree induced by $\{x \in V(T) \mid v \in B_x\}$ is rooted.
    As $vw \in E(G)$ for all $w \in B_x$, we obtain $B_x \subseteq B_{x'}$, so $|B_{x'}|\geq |B_x|+1=t+2$, a contradiction.
    We deduce that
    \[
    |\{X \subseteq V(G) \mid \exists v \in N(u), v>u, X \subseteq N_{<u}(v)\}| \leq 2^t-1
    \]
    and so
    \[
    |N_{<}(u)| + 
        \log \left(|\{X \subseteq V(G) \mid \exists v \in N(u), v>u, X \subseteq N_{<u}(v)\}|\right) \leq t + \log(2^t-1) <2t.
    \]
    If $N_{>}(u) = \emptyset$, then $|N_{<}(u)| \leq t \leq 2t$.
    Hence $G$ is $2t$-strongly-degenerate and the result follows from Theorem~\ref{thm:generic_greedy_argument}.   
\end{proof}

For the  lower bound, observe that $M(\tw\leq t)\geq t$ can be easily proven by combining Theorem~\ref{thm:multipartite} applied to complete graph $K_t$ with the well-known fact that $\tw(K_t)=t-1$ for every $t\geq 2$. 
Somewhat surprisingly, it turns out that complete graphs are not tight examples. 
Indeed, we slightly improve this lower bound, showing that $M(\tw\leq t)\geq t+2$ for every $t\geq 2$. 
First we need a few well-known results about treewidth.

\begin{proposition}\label{prop:some_prop_of_tw}
\leavevmode
\begin{enumerate}[label=(\roman*)]
    \item Let $G$ be a graph of treewidth at least $2$ and let $u,v,w \in V(G)$ with $uv,uw,vw\in E(G)$ and $d_G(u)=2$. Then $\tw(G)=tw(G-u)$. \label{item:twtriangle}
    \item Let $G_1,G_2$ be graphs, $v_i \in V(G_i)$ for $i \in [2]$, and let $G$ be obtained from $G_1$ and $G_2$ by identifying $v_1$ and $v_2$. 
    Then $\tw(G)=\max\{\tw(G_1),\tw(G_2)\}$. \label{item:twmerge}
    \item Let $G$ be a graph and $v \in V(G)$. Then $\tw(G-v)\geq \tw(G)-1$. Moreover, if $v$ is adjacent to all vertices in $V(G)-v$ in $G$, then equality holds. \label{item:twdelete}
\end{enumerate}
\end{proposition}

We further need the following definition. Given a vector $\mathbf{v}$ in $\mathbb{F}_2^t$ for some positive integer $t$, its {\it support} is the number of its coordinates whose value is $1$.

The proof of the lower bound on $M(\tw\leq t)$ for $t \geq 2$ is by induction on $t$. The following construction for $t=2$ will serve as the base case.

\begin{proposition}\label{proposition:tw2diam4}
    There is a graph $G$ with treewidth $2$ such that $\diam(\mathcal{I}(G))\geq 4$.
\end{proposition}

\begin{proof}
Consider the gadgets $G_1(x,y), G_2(x,y)$ and $G_3(x)$ depicted in Figure~\ref{fig:construction_tw2}. The label functions $\pi_1$, $\pi_2$, $\pi_3$ mapping the edges of $G_1(x,y), G_2(x,y)$, and $G_3(x)$, respectively, to $\mathbb{F}_2$ can also be found in Figure~\ref{fig:construction_tw2}.
For every $i\in [3]$, we say that a family $(\mathbf{u})_{u \in V(G_i)}$ of vectors over $\mathbb{F}_2$ \emph{satisfies} $(G_i,\pi_i)$ if $\pi_i(uv) = \mathbf{u} \cdot \mathbf{v}$ for every $uv \in E(G_i)$.
The following observations can be easily checked.
\begin{itemize}
    \item For every family $(\mathbf{u})_{u \in V(G_1(x,y))}$ of vectors in $\mathbb{F}_2^3$ satisfying $(G_1(x,y),\pi_1)$, we have $\mathbf{x} \neq \mathbf{y}$.
    \item For every family $(\mathbf{u})_{u \in V(G_2(x,y))}$ of vectors in $\mathbb{F}_2^3$ satisfying $(G_2(x,y),\pi_1)$, if $\mathbf{x},\mathbf{y} \neq \mathbf{0}$, then $\mathbf{x} + \mathbf{y} \neq (1,1,1)$.
    \item For every family $(\mathbf{u})_{u \in V(G_3(x))}$ of vectors in $\mathbb{F}_2^3$ satisfying $(G_3(x),\pi_3)$, we have $\mathbf{x} \neq (1,1,1)$.
\end{itemize}

Now consider the graph $G$ with the function $\pi \colon E(G) \to \mathbb{F}_2$ depicted on Figure~\ref{subfig:base_of_the_construction}.
We claim that $\diam(\mathcal{I}(G)) >3$.
Suppose for the sake of a contradiction that there is a family $(\mathbf{u})_{u \in V(G_1(x,y))}$ of vectors in $\mathbb{F}_2^3$ satisfying the constraints $\mathbf{u} \cdot \mathbf{v} = \pi(uv)$ for every edge $uv$.
By the previous observations, we have $(1,1,1) \not\in \{\mathbf{a}, \mathbf{b}, \mathbf{c}, \mathbf{d}\}$, $\mathbf{a} \neq \mathbf{c}$ and $\mathbf{b} \neq \mathbf{d}$.
Since every vertex in $\{a,b,c,d\}$ is incident to an edge $e$ with $\pi(e)=1$, we also have $(0,0,0) \not\in \{\mathbf{a}, \mathbf{b}, \mathbf{c}, \mathbf{d}\}$.

By the above observations, we obtain that the support of $\mathbf{b}$ is either $1$ or $2$.
If $\mathbf{b}$ has support $1$, say $\mathbf{b}=(1,0,0)$, then $\mathbf{c} \in \{(1,1,0),(1,0,1),(1,0,0)\}$ as $\pi(bc)=1$ and $\mathbf{c}\neq (1,1,1)$. If $\mathbf{c}=(1,0,0)$, we obtain $0=\pi(bd)=\mathbf{b} \cdot \mathbf{d}=\mathbf{c} \cdot \mathbf{d}=\pi(cd)=1$, a contradiction. This yields $\mathbf{c} \in \{(1,1,0),(1,0,1)\}$ and so $\mathbf{c}$ has support $2$.
Symmetrically, if $\mathbf{c}$ has support 1, then $\mathbf{b}$ has support $2$.
We conclude that either $\mathbf{b}$ or $\mathbf{c}$ has support $2$.
Without loss of generality, suppose that $\mathbf{b}$ has support $2$ and $\mathbf{b} = (1,1,0)$.
Since $\mathbf{b} \cdot \mathbf{d}=0$, we have $\mathbf{d} \in \{(0,0,0),(0,0,1),(1,1,0),(1,1,1)\}$.
However, by the above observations, we have $\mathbf{d} \notin \{(0,0,0),(1,1,1)\}$, $\mathbf{b}\neq \mathbf{d}$, and $\mathbf{b} + \mathbf{d} \neq (1,1,1)$, and so $\mathbf{d} \notin \{(0,0,0),(0,0,1),(1,1,0),(1,1,1)\}$, a contradiction.

Thus $\diam(\mathcal{I}(G)) >3$ by Observation~\ref{obs:characterization_with_vectors}.
Further, as $G$ is not a tree and repeatedly applying Proposition~\ref{prop:some_prop_of_tw}~\ref{item:twtriangle} and~\ref{item:twmerge}, we obtain that $G$ has treewidth $2$. This proves the proposition.
\end{proof}

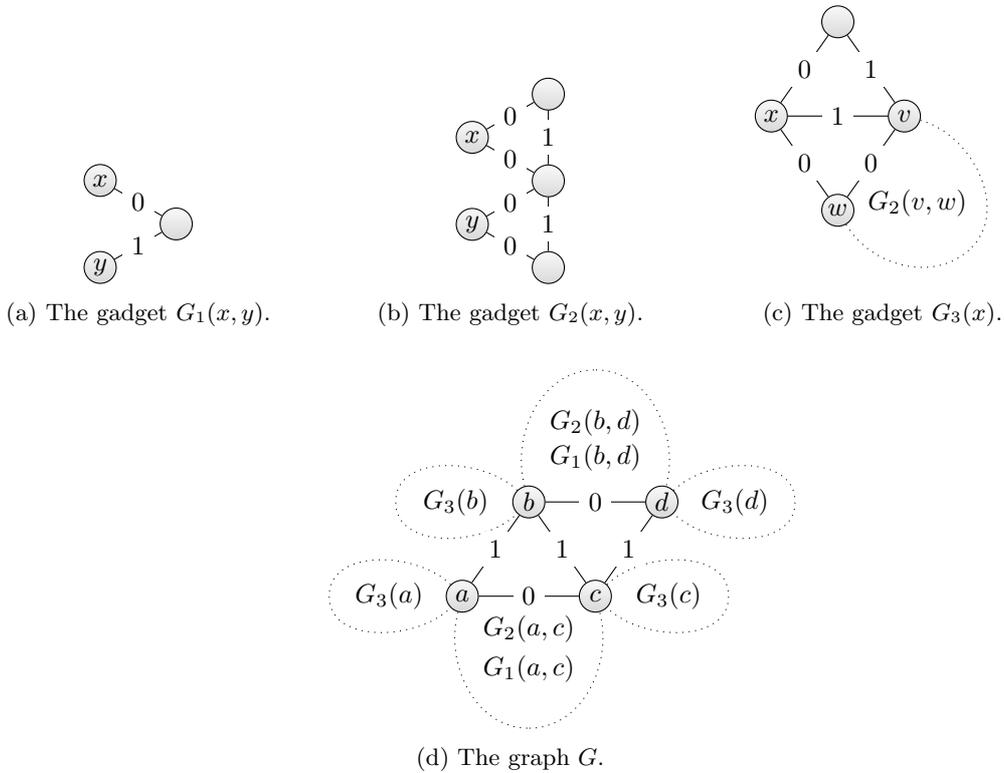
\begin{figure}[hbtp]
    \begin{center}
        \begin{subfigure}[b]{0.3\textwidth}
        \centering
        \begin{tikzpicture}
            \tikzstyle{vertex}=[circle,draw, top color=gray!5, 
	    bottom color=gray!30, minimum size=12pt, scale=1, inner sep=0.5pt]
            \node[vertex] (x) at (0,0.577) {$x$};
            \node[vertex] (y) at (0,-0.577) {$y$};
            \node[vertex] (u) at (1,0) {};

            \draw (x) -- node[midway,fill=white]{$0$} (u);
            \draw (y) -- node[midway,fill=white]{$1$} (u);
          \end{tikzpicture}
      \caption{The gadget $G_1(x,y)$.}
      \end{subfigure}
      \begin{subfigure}[b]{0.3\textwidth}
        \centering
        \begin{tikzpicture}
            \tikzstyle{vertex}=[circle,draw, top color=gray!5, 
	    bottom color=gray!30, minimum size=12pt, scale=1, inner sep=0.5pt]
            \node[vertex] (x) at (0,0.577) {$x$};
            \node[vertex] (y) at (0,-0.577) {$y$};
            \node[vertex] (u) at (1,0) {};
            \node[vertex] (w) at (1,1.15) {};
            \node[vertex] (w') at (1,-1.15) {};

            \draw (x) -- node[midway,fill=white]{$0$} (u);
            \draw (y) -- node[midway,fill=white]{$0$} (u);
            \draw (x) -- node[midway,fill=white]{$0$} (w);
            \draw (w) -- node[midway,fill=white]{$1$} (u);
            \draw (y) -- node[midway,fill=white]{$0$} (w');
            \draw (u) -- node[midway,fill=white]{$1$} (w');
          \end{tikzpicture}
      \caption{The gadget $G_2(x,y)$.}
    \end{subfigure}
    \begin{subfigure}[b]{0.3\textwidth}
        \centering
        \begin{tikzpicture}[scale=1.75]
            \tikzstyle{vertex}=[circle,draw, top color=gray!5, 
	    bottom color=gray!30, minimum size=12pt, scale=1, inner sep=0.5pt]
            \node[vertex] (x) at (0,0) {$x$};
            \node[vertex] (a) at (1,0) {$v$};
            \node[vertex] (w) at (0.5,0.714) {};
            \node[vertex] (w') at (0.5,-0.714) {$w$};

            \node[minimum size=0pt, inner sep=0pt] (temp) at (1.5, -1) {};
            \node (label) at (1.1,-0.66) {$G_2(v,w)$};

            \draw (x) -- node[midway,fill=white]{$1$} (a);
            \draw (x) -- node[midway,fill=white]{$0$} (w);
            \draw (x) -- node[midway,fill=white]{$0$} (w');
            \draw (a) -- node[midway,fill=white]{$1$} (w);
            \draw (a) -- node[midway,fill=white]{$0$} (w');
            \draw[dotted] (a) to[out=-15, in=45] (temp);
            \draw[dotted] (w') to[out=-60, in=-135] (temp);
          \end{tikzpicture}
      \caption{The gadget $G_3(x)$.}
    \end{subfigure}

    \vspace{5mm}
    
    \begin{subfigure}[b]{0.3\textwidth}
        \centering
        \begin{tikzpicture}[scale=1.75]
            \tikzstyle{vertex}=[circle,draw, top color=gray!5, 
	    bottom color=gray!30, minimum size=12pt, scale=1, inner sep=0.5pt]
            \node[vertex] (a) at (0,0) {$a$};
            \node[vertex] (c) at (1,0) {$c$};
            \node[vertex] (b) at (0.5,0.714) {$b$};
            \node[vertex] (d) at (1.5,0.714) {$d$};

            \node[minimum size=0pt, inner sep=0pt] (temp) at (0.5, -1) {};
            \draw[dotted] (a) to[out=-105, in=-180] (temp);
            \draw[dotted] (temp) to[out=0, in=-75] (c);
            \node (label) at (0.5,-0.25) {$G_2(a,c)$};
            \node (labelb) at (0.5,-0.55) {$G_1(a,c)$};

            \node[minimum size=0pt, inner sep=0pt] (temp2) at (1, 1.714) {};
            \draw[dotted] (b) to[out=105, in=180] (temp2);
            \draw[dotted] (temp2) to[out=0, in=75] (d);
            \node (label2) at (1,1.05) {$G_1(b,d)$};
            \node (label22) at (1,1.314) {$G_2(b,d)$};

            \node[minimum size=0pt, inner sep=0pt] (temp3) at (-1, 0) {};
            \draw[dotted] (a) to[out=135, in=90] (temp3);
            \draw[dotted] (temp3) to[out=-90, in=-135] (a);
            \node (label3) at (-0.55,0) {$G_3(a)$};

            \node[minimum size=0pt, inner sep=0pt] (temp4) at (-0.5, 0.714) {};
            \draw[dotted] (b) to[out=135, in=90] (temp4);
            \draw[dotted] (temp4) to[out=-90, in=-135] (b);
            \node (label32) at (-0.05,0.714) {$G_3(b)$};

            \node[minimum size=0pt, inner sep=0pt] (temp5) at (2, 0) {};
            \draw[dotted] (c) to[out=45, in=90] (temp5);
            \draw[dotted] (temp5) to[out=-90, in=-45] (c);
            \node (label33) at (1.55,0) {$G_3(c)$};

            \node[minimum size=0pt, inner sep=0pt] (temp6) at (2.5, 0.714) {};
            \draw[dotted] (d) to[out=45, in=90] (temp6);
            \draw[dotted] (temp6) to[out=-90, in=-45] (d);
            \node (label34) at (2.05,0.714) {$G_3(d)$};

            \draw (a) -- node[midway,fill=white]{$1$} (b);
            \draw (b) -- node[midway,fill=white]{$1$} (c);
            \draw (c) -- node[midway,fill=white]{$1$} (d);
            \draw (a) -- node[midway,fill=white]{$0$} (c);
            \draw (b) -- node[midway,fill=white]{$0$} (d);
            
          \end{tikzpicture}
      \caption{The graph $G$.}
      \label{subfig:base_of_the_construction}
    \end{subfigure}
    \end{center}
    \caption{Construction of a graph with treewidth $2$ and inversion diameter $4$.}
    \label{fig:construction_tw2}
\end{figure}

We now show that more generally there are graphs of treewidth $t$ and inversion diameter $t+2$, for every $t \geq 2$.
First we need the following lemma.
Given a vector $\mathbf{u} \in \mathbb{F}_2^t$, we denote by $\mathbf{u}^\bot$ the subspace $\{\mathbf{v} \in \mathbb{F}_2^t \mid \mathbf{u} \cdot \mathbf{v} = 0\}$.

In order to extend the lower bound for arbitrary values of $t$, we need the following result. For some positive integer $t$, a basis $(\mathbf{x}_1,\ldots,\mathbf{x}_t)$ of $\mathbb{F}_2^t$ is called {\it orthogonal} if $\mathbf{x_i} \cdot \mathbf{x_i}=1$ holds for all $i \in [t]$ and $\mathbf{x_i} \cdot\mathbf{x_j}=0$ holds for all distinct $i,j \in [t]$.
\begin{lemma}\label{lemma:find_u_st_ubot_has_an_orthogonal_basis}
Let $t$ be an integer such that $t \geq 2$. Let $\mathbf{u}$ be a vector in $\mathbb{F}_2^t$ whose support is odd and different from $t$.  
Then the subspace $\mathbf{u}^\bot$ 
has an orthogonal basis.
In particular, there are at most $2^{t-1}+1$ vectors $\mathbf{u}$ in $\mathbb{F}_2^t$ such that $\mathbf{u}^\bot$ does not admit an orthogonal basis.
\end{lemma}

\begin{proof}
    Let $\ell$ be the support of $\mathbf{u}$.
    Without loss of generality, $\mathbf{u}=(1,\dots,1,0,\dots,0)$.
    Let $J_\ell$ be the $(\ell+1) \times (\ell+1)$ matrix all of whose entries are $1$ except those on the anti-diagonal, that is $(J_{\ell})_{i,j}=0$ if and only if $i=\ell+2-j$ for every $i,j \in [\ell+1]$.
    Since $\ell$ is odd, we have $J_\ell ^ 2 = I_{\ell +1}$.
    Then the matrix
    \[
    M = 
    \left(
    \begin{array}{c c}
        J_\ell & 0 \\
        0 & I_{t - \ell-1} \\
    \end{array}
    \right)
    \]
    is such that $M^2 = I_t$ and its first column is $\mathbf{u}$.
    Hence the columns of $M$ form an orthogonal basis of $\mathbb{F}_2^t$ containing $\mathbf{u}$.
    Therefore the family of columns obtained from the columns of $M$ by deleting the first one forms an orthogonal basis of $\mathbf{u}^\bot$.
\end{proof}

We are now ready to prove the following restatement of the fact that $M(\tw \leq t)\geq t+2$.
\begin{theorem}\label{thm:lower_bound_for_tw}
    For every $t \geq 2$, there is a graph $G_t$ of treewidth $t$ such that $\diam(\mathcal{I}(G)) \geq t+2$.
\end{theorem}

\begin{proof}
    We proceed by induction on $t$.
    For $t=2$, the result follows from Proposition~\ref{proposition:tw2diam4}.
    Now assume $t \geq 3$.
    By the induction hypothesis and Observation~\ref{obs:characterization_with_vectors}, there is a graph $G_{t-1}$ with $\tw(G_{t-1})=t-1$ and a function $\pi_{t-1} \colon E(G_{t-1}) \to \mathbb{F}_2$ such that there is no family of vectors $(\mathbf{z})_{z \in V(G_{t-1})}$ in $\mathbb{F}_2^{t}$ such that $\pi_{t-1}(yz) = \mathbf{y} \cdot \mathbf{z}$ for every edge $yz \in E(G_{t-1})$.
    
    We now construct $G_t$ as follows.
    We first let $V(G_t)$ contain $t$ vertices $v_1, \dots, v_t$.
    Then, for every $\mathbf{x} \in \mathbb{F}_2^t$, we add a vertex $w_\mathbf{x}$ adjacent to all vertices in $\{v_1, \dots, v_t\}$
    and we set $\pi(v_i w_\mathbf{x}) = \mathbf{x}_i$ for every $i \in [t]$.
    Moreover, for every $i \in [t]$ and $\mathbf{x}\in \mathbb{F}_2^t$, we add a vertex $a_{i,\mathbf{x}}$ with adjacent to $v_i$ and $w_\mathbf{x}$, and set $\pi(a_{i,\mathbf{x}} v_i)=0$ and $\pi(a_{i,\mathbf{x}} w_\mathbf{x})=1$.
    Finally, for every $b \in B$, we add a copy $G_b$ of $G_{t-1}$ with the labels $\pi_{t-1}$ and, for all $z \in V(G_b)$, we add the edge $b z$ and set $\pi(bz)=0$ where $B=\{v_1, \dots, v_t\} \cup \{w_\mathbf{x}\}_{\mathbf{x} \in \mathbb{F}_2^t}$.

    We first show that the treewidth of $G_t$ is as desired.
    \begin{claim}\label{twcalc}
        $\tw(G_t)=t$.
    \end{claim}
    \begin{subproof}
        As $\tw(G_{t-1})=t-1$ and by Proposition~\ref{prop:some_prop_of_tw}~\ref{item:twdelete}, we have $\tw(G_t[V(G_b)\cup \{b\}])=t$ for all $b \in B$. 
        By Proposition~\ref{prop:some_prop_of_tw}~\ref{item:twmerge}, we obtain $\tw(G_t)\geq t$ and $\tw(G_t) = \max\{t, \tw(G_t-\bigcup_{b \in B}V(G_b))\}$. 
        Next, as $G_t[B]$ is not a tree, we obtain that $\tw(G_t-\bigcup_{b \in B}V(G_b)) = \tw(G_t[B])$ by Proposition~\ref{prop:some_prop_of_tw}~\ref{item:twtriangle}. 
        Then, by Proposition~\ref{prop:some_prop_of_tw}~\ref{item:twdelete}, we obtain $\tw(G_t[B])\leq \tw(G_t[B-\{v_1,\ldots,v_t\}])+t$. As $B-\{v_1,\ldots,v_t\}$ is an independent set in $G_t$, we obtain $\tw(G_t[B])\leq t$. This yields $\tw(G_t) = \max\{t, \tw(G_t-\bigcup_{b \in B}V(G_b))\} = t$.
    \end{subproof}

    Suppose for the sake of a contradiction that there is a family $(\mathbf{u})_{u \in V(G_t)}$ of vectors in $\mathbb{F}_2^{t+1}$ such that $\pi(yz) = \mathbf{y} \cdot \mathbf{z}$
    for every edge $yz$ of $G_t$.
    \begin{claim}
        The vectors $(\mathbf{b})_{b \in B}$ are pairwise distinct.
    \end{claim}\label{dist}
    
    \begin{subproof}
        First consider some distinct $\mathbf{x},\mathbf{x}' \in \mathbb{F}_2^{t}$. Then there exists some $i \in [t]$ such that $\mathbf{x}_i \neq \mathbf{x}'_i$.
        If $\mathbf{w}_\mathbf{x}=\mathbf{w}_{\mathbf{x}'}$, we obtain 
        $\pi(w_\mathbf{x} v_i)=\mathbf{w}_\mathbf{x} \mathbf{v}_i=\mathbf{w}_{\mathbf{x}'}\mathbf{v}_i=\pi(w_{\mathbf{x}'}v_i)$, a contradiction to the construction.
        
        Now consider some distinct $i,j \in [t]$ and some $\mathbf{x} \in \mathbb{F}_2^{t}$ with $\mathbf{x}_i \neq \mathbf{x}_j$. If $\mathbf{v}_i=\mathbf{v}_{j}$, 
        we obtain $\pi(w_\mathbf{x} v_i)=\mathbf{w}_\mathbf{x} \mathbf{v}_i=\mathbf{w}_{\mathbf{x}}\mathbf{v}_j=\pi(w_\mathbf{x} v_j)$, a contradiction to the construction.

        Finally consider some  $i \in [t]$ and $\mathbf{x} \in \mathbb{F}_2^{t}$. If $\mathbf{v}_i=\mathbf{w}_{\mathbf{x}}$, 
        we obtain $\pi(v_ia_{i,\mathbf{x}})=\mathbf{v}_i\mathbf{a}_{i,\mathbf{x}}=\mathbf{w}_{\mathbf{x}}\mathbf{a}_{i,\mathbf{x}}=\pi(w_xa_{i,\mathbf{x}})$, 
        a contradiction to the construction.
    \end{subproof}
    As $|B|=2^t+t>2^t+1$, by Claim~\ref{dist} and by Lemma~\ref{lemma:find_u_st_ubot_has_an_orthogonal_basis}, there exists $b \in B$
    such that $\mathbf{b}^\bot$ has an orthogonal basis $\mathbf{x}_1, \dots, \mathbf{x}_{t}$.
    Now, for every $z \in V(G_b)$, we have $\mathbf{z} \in \mathbf{b}^\bot$ and so there are scalars $z_1, \dots, z_t$ in $\mathbb{F}_2$ such that $\mathbf{z} = \sum_{i=1}^{t} z_i \, \mathbf{x}_i$.
    Let $\mathbf{z}' = (z_1, \dots, z_{t}) \in \mathbb{F}_2^t$ for every $z \in V(G_b)$.
    Then, as $\mathbf{x}_1, \dots, \mathbf{x}_{t}$ is an orthogonal basis of $\mathbf{b}^\bot$, for every edge $yz \in E(G_b)$, we have
    \begin{align*}
        \mathbf{y}' \cdot \mathbf{z}'
        &=\sum_{i \in [t]}y_iz_i\\
        &=\sum_{i,j \in [t], i \neq j}y_iz_j \, \mathbf{x}_i \cdot \mathbf{x}_j+\sum_{i \in [t]}y_iz_i \, \mathbf{x}_i \cdot \mathbf{x}_i\\
        &=\sum_{i\in [t]}y_i\mathbf{x}_i \cdot \sum_{j \in [t]}z_j\mathbf{x}_j\\
        &=\mathbf{y} \cdot \mathbf{z}\\
        &=\pi(yz).
    \end{align*}
    This contradicts the assumption on $G_{t-1}$.

    Thus, by Observation~\ref{obs:characterization_with_vectors}, we have $\diam(\mathcal{I}(G_t)) > t+1$.
    Since $G_t$ has treewidth $t$ by Claim~\ref{twcalc}, this proves the theorem.    
\end{proof}

\section{\texorpdfstring{$K_t$}{Kt}-minor-free and planar graphs}\label{sec:minor_closed}
This section is dedicated to proving the results guaranteeing small inversion diameter for graphs contained in certain minor-closed classes.
It was proven by van den Heuvel at al. \cite{van2017generalised} that for every $t\geq 4$ and every $K_t$-minor-free graph $G$, we have $\chi_s(G) \leq 5\binom{t}{2}(t-3)$. 
Together with Corollary~\ref{cor:chi_s}, we obtain the following result.

\begin{corollary}\label{c_2}
    For every proper minor-closed class $\mathcal{G}$ of graphs, there exists a constant $\alpha_\mathcal{G}$ such that $\diam ({\mathcal{I}}(G))\leq \alpha_\mathcal{G}$ for every $G \in \mathcal{G}$.
\end{corollary}

Hence proper minor-closed families have bounded inversion diameter. In the remainder of this section, we prove more precise bounds for some particular minor-closed classes. 

\subsection{\texorpdfstring{$K_t$}{Kt}-minor-free graphs}

In this section, we deal with $K_t$-minor-free graphs. 
To do so, we will use known bounds on strong colouring numbers. By the {\bf length} of a path, we refer to its number of edges.
Given a graph $G$, an ordering $<$ of its vertices, an integer $r \geq 0$, and $u,v \in V(G)$, we say that  $v$ is {\bf $r$-strongly reachable from $u$ in $G$ for $<$}
if there is a path $P$ from $u$ to $v$ of length at most $r$ such that $V(P) \cap \{w \in V(G) \mid w\leq u\} = \{u,v\}$.
Note that $u$ is $r$-strongly reachable from $u$.
We denote by $\sreach_r[G,<,u]$ the set of all vertices in $G$ $r$-strongly reachable from $u$ in $G$ for $<$.

\begin{theorem}[\cite{van2017generalised}]\label{thm:vdHeuvel_et_al_Kt_minor_free}
    For every integer $t \geq 3$ and for every $K_t$-minor-free graph $G$, there is an ordering $<$ of $V(G)$ such that
    \[
    |\sreach_r[G,<,u]| \leq \binom{t-1}{2}(2r+1)
    \]
    for every nonnegative integer $r$ and every vertex $u \in V(G)$.
\end{theorem}

Note that in~\cite{van2017generalised}, this theorem is stated without specifying that $<$ does not depend on $r$. However, this fact follows from the proof.

Together with Theorem~\ref{thm:generic_greedy_argument}, Theorem~\ref{thm:vdHeuvel_et_al_Kt_minor_free} implies the
following.

\begin{corollary}\label{cor:Kt-free}
    Let $t$ be an integer with $t \geq 3$. For every $K_t$-minor-free graph $G$,
    \[
    \diam(\mathcal{I}(G)) \leq 8\binom{t-1}{2}.
    \]
\end{corollary}

\begin{proof}
    Let $<$ be an ordering of $V(G)$ as in Theorem~\ref{thm:vdHeuvel_et_al_Kt_minor_free}.
    For every vertex $u$ of $G$ we have
    \begin{enumerate}
        \item $|N_<(u) | \leq 3\binom{t-1}{2}-1$, and
        \item $|\{X \subseteq V(G) \mid \exists~v \in N(u), u<v, X \subseteq \{ w \in N(v) \mid w<u\}\}| \leq 2^{5\binom{t-1}{2}}$.
    \end{enumerate}
    Hence $G$ is strongly $8\binom{t-1}{2}$-degenerate and so $\diam(\mathcal{I}(G)) \leq 8\binom{t-1}{2}$ by Theorem~\ref{thm:generic_greedy_argument}.
\end{proof}

\subsection{Planar graphs}\label{sec:planar}

\subsubsection{General bound}

A planar graph is $K_5$-minor-free, so by Corollary~\ref{cor:Kt-free}, its inversion diameter is at most 48. A planar graph has  acyclic chromatic number at most 5 as shown by Borodin~\cite{Bor79}. Thus, by Corollary~\ref{cor:L-acyclic-2}, its inversion diameter is at most 18.
We shall improve on this bound by showing that $\diam(\mathcal{I}(G)) \leq 12$ for every planar graph $G$.
This bound is obtained by using
the following result establishing the existence of an ordering $<$ such that $\sreach_r[G,<,u]$ has small size for every $u$.\footnote{As in Theorem~\ref{thm:vdHeuvel_et_al_Kt_minor_free}, the statement in~\cite{van2017generalised} of Theorem~\ref{thm:vdHeuvel_et_al_scol} does not specify that the ordering $<$ is the same for every $r$.
However, this fact follows from the proof.}

\begin{theorem}[\cite{van2017generalised}]\label{thm:vdHeuvel_et_al_scol}
    For every planar graph $G$, there is an ordering $<$ of $V(G)$
    such that
    \[
        |\sreach_r[G,<,u]| \leq 5r+1
    \]
    for every vertex $u \in V(G)$ and every nonnegative integer $r$.
\end{theorem}

\begin{corollary}\label{cor:planar_12}
    $\diam(\mathcal{I}(G)) \leq 12$ for every planar graph $G$.
\end{corollary}

\begin{proof}
    Let $<$ be an ordering of $V(G)$ as in  Theorem~\ref{thm:vdHeuvel_et_al_scol}.
    Let $u \in V(G)$.
    Let $A = \{u\} \cup \bigcup_{v \in N_{>}(u)}N_{<u}(v) \subseteq \sreach_2[G,<,u]$ and $B = N_{>}(u)$.
    Consider the bipartite graph $H$ with bipartition $(A,B)$ and with edges inherited from $G$.
    Let $H'=(A,B',E')$ be the graph obtained from $H$ by removing vertices of degree at most $2$ in $B$.
    Observe that $H'$ is a subgraph of $G$ and hence planar. 
    Moreover, as $H'$ is bipartite, we obtain that the girth of $H'$ is at least 4.
    It hence follows from Euler's formula that $m(H') = \sum_{a \in A} d_{H'}(a) = \sum_{b \in B'}d_{H'}(b) \leq 2(|A|+|B'|)-4$.
    Moreover, every vertex in $B'$ has degree at most $5$, and has $u$ as a neighbour.
    For $i\in \{2,3,4\}$, let $d_i$ be the number of vertices in $B'$ of degree $i+1$ in $H'$. 
    Then we have
    \[
    |\{X \subseteq V(G) \mid \exists~v \in N(u), u<v, X \subseteq N_{<u}(v)\}| \leq 1+\binom{10}{1} + \binom{2}{2} d_2 + \binom{3}{2\leq\cdot\leq 3}d_3 + \binom{4}{2\leq\cdot\leq 4}d_4 
    \]
    using the notation $\binom{n}{k_1 \leq \cdot \leq k_2} = \sum_{k_1 \leq k \leq k_2} \binom{n}{k}$.
    Moreover,
    \[
    3d_2+4d_3+5d_4 \leq m(H') \leq 2(|A|+|B'|)-4 \leq 2(11+d_2+d_3+d_4)-4.
    \]
    We deduce
    \[
        d_2 + 2d_3 + 3d_4 \leq 18
    \]
    and so 
    \[
        d_2 + 4d_3 + 11d_4 \leq \frac{11}{3}(d_2 + 2d_3 + 3d_4) \leq 66.
    \]
    Finally we have
    \[
        |\{X \subseteq V(G) \mid \exists~v \in N(u), u<v, X \subseteq N_{<u}(v)\}| \leq 11 + d_2 + 4d_3 + 11 d_4 \leq 77.
    \]
    Therefore,
    \begin{enumerate}
        \item $|\{v \in N(u) \mid v<u\}| \leq 5$ and
        \item $|\{X \subseteq V(G) \mid \exists~v \in N(u), u<v, X \subseteq N_{<u}(v)\}| \leq 77 < 2^7$,
    \end{enumerate}
    and it follows that $G$ is $12$-strongly-degenerate and so $\diam(\mathcal{I}(G)) \leq 12$ by Theorem~\ref{thm:generic_greedy_argument}.
\end{proof}

For the lower bound on the maximum inversion diameter of a planar graph, we executed a computer search. 
As a result, we found a planar graph whose inversion diameter is at least 5. 

\begin{proposition}\label{prop:planar5}
    If $G$ if the planar graph depicted on Figure~\ref{fig:planar5}, then $\diam(\mathcal{I}(G)) \geq 5$.
\end{proposition}
\begin{figure}[hbtp]
    \centering
    \begin{tikzpicture}
        \tikzstyle{vertex}=[circle,draw, top color=gray!5, bottom color=gray!30, minimum size=12pt, scale=1, inner sep=0.5pt]
        \node[vertex] (7) at (-3,0) {};
        \node[vertex] (1) at (-1,0) {};
        \node[vertex] (2) at (1,0) {};
        \node[vertex] (3) at (3,0) {};

        \node[vertex] (6) at (-2,1.73) {};
        \node[vertex] (8) at (0,1.73) {};
        \node[vertex] (4) at (2,1.73) {};

        \node[vertex] (5) at (0,3.5) {};

        \node[vertex] (0) at (0,-2.75) {};

        \draw (7) to node[midway, fill=white] {$0$} (1);
        \draw (1) to node[midway, fill=white] {$1$} (2);
        \draw (2) to node[midway, fill=white] {$0$} (3);
        \draw (7) to node[midway, fill=white] {$1$} (6);
        \draw (6) to node[midway, fill=white] {$0$} (1);
        \draw (1) to node[midway, fill=white] {$0$} (8);
        \draw (8) to node[midway, fill=white] {$0$} (2);
        \draw (2) to node[midway, fill=white] {$0$} (4);
        \draw (4) to node[midway, fill=white] {$1$} (3);
        \draw (6) to node[midway, fill=white] {$0$} (8);
        \draw (8) to node[midway, fill=white] {$0$} (4);
        \draw (6) to node[midway, fill=white] {$1$} (5);
        \draw (8) to node[midway, fill=white] {$1$} (5);
        \draw (4) to node[midway, fill=white] {$0$} (5);

        \draw (0) to node[midway, fill=white] {$1$} (7);
        \draw (0) to node[midway, fill=white] {$0$} (1);
        \draw (0) to node[midway, fill=white] {$0$} (2);
        \draw (0) to node[midway, fill=white] {$0$} (3);

        \draw[out=170, in=190, looseness=1.65] (0) to node[midway, fill=white] {$0$} (6);
        \draw[bend left=90, looseness=2.25] (0) to node[midway, fill=white] {$0$} (5);
        \draw[out=10, in=-10, looseness=1.65] (0) to node[midway, fill=white] {$0$} (4);
    \end{tikzpicture}
    \caption{A planar graph $G$ with inversion diameter at least $5$.
        More precisely, if $\pi \colon E(G) \to \mathbb{F}_2$ is the labelling function depicted, then there is no family $(\mathbf{u})_{u \in V(G)} \in \mathbb{F}_2^4$ such that $\pi(uv) = \mathbf{u} \cdot \mathbf{v}$ for every $uv \in E(G)$.}
    \label{fig:planar5}
\end{figure}

\subsubsection{Planar graphs of large girth}

We now examine to which extent the upper bounds on the inversion diameter of a planar graph can be improved assuming a large girth. The {\bf girth} of a graph is defined to be the length of its shortest cycle. 
Borodin et al.~\cite{BKW99} proved that a planar graph has acyclic chromatic number at most $4$ (resp. $3$) if it has girth at least $5$ (resp. $7$).
With Corollary~\ref{cor:L-acyclic-2}, this yields the following.
\begin{corollary}\label{thm:planar-girth}
    Let $G$ be a planar graph.
\begin{itemize}
    \item If $G$ has girth at least $5$, then $\diam(\mathcal{I}(G)) \leq 10$.
    \item If $G$ has girth at least $7$, then $\diam(\mathcal{I}(G)) \leq 6$.
    \end{itemize}

\end{corollary}

In the remainder of this subsection, we show that sufficiently large girth guarantees that a planar graph has inversion diameter at most 3. 

\begin{theorem}\label{thm:planar-3}
    Let $G$ be a planar graph.
    If $G$ has girth at least $8$, then $\diam(\mathcal{I}(G)) \leq 3$.
\end{theorem}

Inversion diameter 3 is the smallest that 
a girth condition can guarantee for a planar graph as shown by the following 
proposition. 

\begin{proposition}\label{prop:planar_even_girth_inversion_3}
    For every even positive integer $g$, there is a bipartite planar graph $G$ with girth $g$ such that 
    $\diam({\mathcal{I}}(G)) \geq 3$.
\end{proposition}

\begin{proof}
    Let $G$ be the graph obtained from a cycle on $g$ vertices $C=v_1v_2 \cdots v_{g}$ by adding for every vertex $v\in V(C)$ a new vertex $w_v$ adjacent uniquely to $v$. By construction, we have $d_G(v)\geq 3$ for all $v \in V(C)$. We hence obtain $\diam({\mathcal{I}}(G)) \geq 3$ by Lemma~\ref{lem:aumoins3}. 
\end{proof}

\medskip

In fact, we shall prove a result  on graphs of maximum average degree at most $2 + \frac{8}{11}$ 
which implies Theorem~\ref{thm:planar-3} using the following well-known bound on the maximum average degree of planar graphs of large girth.

\begin{lemma}[Folklore]\label{lem:Ad-girth}
    If $G$ is a planar graph with girth at least $g$, then $\Mad(G) < 2+\frac{4}{g-2}$.
\end{lemma}

The proof of this lemma, which is a direct application of Euler's formula, is left to the reader.
By Lemma~\ref{lem:Ad-girth}, in order to prove Theorem~\ref{thm:planar-3}, it is sufficient to prove that if $\Mad(G) < 2+\frac{2}{3}$ then $\diam(\mathcal{I}(G)) \leq 3$.
We shall in fact prove the following stronger result.
\begin{theorem}\label{thm:3decharg}
    Let $G$ be a graph.
    If $\Mad(G) \leq 2 + \frac{8}{11}$, then $\diam(\mathcal{I}(G)) \leq 3$.
\end{theorem}

Let $G$ be a graph, and $k$ a positive integer.
A {\bf $k$-vertex} (resp. $(\geq k)$-vertex, $(\leq k)$-vertex) is a vertex with degree $k$ (resp. at least $k$, at most $k$).


Let $G$ be a graph and let $\pi\colon E(G) \to \mathbb{F}_2$.
For every positive integer $d$, a {\bf $(G,\pi)$-realisation} of dimension $d$ 
is a family $(\mathbf{u})_{u \in V(G)}$ of vectors 
in $\mathbb{F}_2^d$ such that for every edge $uv$ of $G$, we have $\pi(uv) = \mathbf{u} \cdot \mathbf{v}$.
A $(G,\pi)$-realisation is {\bf strict} if $\mathbf{u} \neq \mathbf{0}$ for every $u \in V(G)$.
Hence, Observation~\ref{obs:characterization_with_vectors} can be rephrased as follows: $\diam(\mathcal{I}(G)) \leq d$ if and only if 
for every $\pi \colon E(G) \to \mathbb{F}_2$, there is a $(G,\pi)$-realisation of dimension $d$.

If $H$ is a subgraph of $G$, we abbreviate an $(H,\pi\vert_{E(H)})$-realisation to an {\it $(H,\pi)$-realisation.} We further say that a $(G,\pi)$-realisation {\bf extends} an $(H,\pi)$-realisation if the two realisations coincide on $V(H)$.

The proof of Theorem~\ref{thm:3decharg} uses the Discharging Method. However, instead of choosing a minimum counterexample to our main result, we choose a minimum counterexample to the slightly stronger result that no strict $(G,\pi)$-realisation exists for the given instance. We first present a collection of structural properties. As some of these results need to be proved in a slightly stronger form in order to be reused later, we give them in a form describing when certain realisations can be extended rather than explicitly speaking about a minimum counterexample. These structural results are contained in Lemmas~\ref{lem:1-vertex} to~\ref{lem:deficient}. After, we give the main proof of Theorem~\ref{thm:3decharg} in the form of a discharging procedure.

We first deal with 1-vertices.
\begin{lemma}\label{lem:1-vertex}
    Let $G$ be a graph, let $\pi \colon  E(G) \to \mathbb{F}_2$ and let $v$ a $1$-vertex in $G$. 
    Any strict $(G-v,\pi)$-realisation of dimension $3$ can be extended into three different strict $(G,\pi)$-realisations of dimension~$3$.
\end{lemma}

\begin{proof}
    Let $u$ be the unique neighbour of $v$.
    Since $\mathbf{u} \neq \mathbf{0}$,
    the solutions $\mathbf{v}$ of the equation $\mathbf{u} \cdot \mathbf{v} = \pi(uv)$ form an affine space of dimension $3-1 = 2$.
    Thus, there are at least $2^2=4$ vectors $\mathbf{v}$ such that $\mathbf{u} \cdot \mathbf{v} = \pi(uv)$.
    Since at most one of them is $\mathbf{0}$, we obtain at least $3$ extensions.
\end{proof}

We now show that the configurations when a realisation cannot be extended to a 2-vertex are very restricted.
\begin{lemma}\label{lem:2-vertex}
    Let $G$ be a graph, let $\pi \colon  E(G) \to \mathbb{F}_2$, and let $v$ be a $2$-vertex in $G$ with neighbours $u_1$ and $u_2$.
    A strict $(G-v,\pi)$-realisation $(\mathbf{u})_{u \in V(G-v)}$ can be extended into a strict $(G,\pi)$-realisation unless
    $\mathbf{u_1} = \mathbf{u_2}$ and $\pi(u_1v) \neq \pi(u_2v)$.
\end{lemma}

\begin{proof}
    If $\mathbf{u_1} \neq \mathbf{u_2}$, then $\mathbf{u_1}, \mathbf{u_2}$ are linearly independent.
    This implies that the affine space of the solutions $\mathbf{v}$ of
    \[
    \left\{
    \begin{array}{r c l}
        \mathbf{u_1} \cdot \mathbf{v} &=& \pi(u_1 v) \\
        \mathbf{u_2} \cdot \mathbf{v} &=& \pi(u_2 v) \\
    \end{array}
    \right.
    \]
    is non empty and has dimension $3-2 = 1$.
    Hence there is a solution $\mathbf{v}$ which is not $\mathbf{0}$.

    Now suppose $\mathbf{u_1} = \mathbf{u_2}$ and $\pi(u_1v) = \pi(u_2v)$.
    Since $\mathbf{u_1} \neq \mathbf{0}$, there is a vector $\mathbf{v}$ such that $\mathbf{u_1} \cdot \mathbf{v} = \mathbf{u_2} \cdot \mathbf{v} = \pi(u_1 v) = \pi(u v_2)$.
    This proves the lemma.
\end{proof}

We are now ready to exclude the existence of two adjacent 2-vertices in a minimum counterexample.
\begin{lemma}\label{lem:3thread}
    Let $G$ be a graph, $\pi \colon  E(G) \to \mathbb{F}_2$, and $v_1, v_2$ two adjacent $2$-vertices in $G$.
    Any strict $(G-\{v_1, v_2\},\pi)$-realisation of dimension $3$ can be extended into a strict $(G,\pi)$-realisation. 
\end{lemma}

\begin{proof}
    Let $(\mathbf{u})_{u \in V(G-\{v_1,v_2\})}$ be a $(G-\{v_1,v_2\}, \pi)$-realisation of dimension $3$.
    Let $u_1$ (resp. $u_2$) be the neighbour of $v_1$ (resp. $v_2$) distinct from $v_2$ (resp. $v_1$).
    By Lemma~\ref{lem:1-vertex}, there are three possible non-zero vectors $\mathbf{v}_1$ that extend the realisation $(\mathbf{u})_{u \in V(G-\{v_1,v_2\})}$ to $G-v_2$.
    So one can assign $\mathbf{v}_1$ so that $\mathbf{v}_1 \neq \mathbf{u}_2$.
    Then, by Lemma~\ref{lem:2-vertex}, there is a choice of $\mathbf{v}_2$ that extends the realisation to $G$.
\end{proof}

We use $N[v]$ to denote the closed neighbourhood of $v$, that is, the set containing $v$ and all its neighbours.

The next result shows that, in a minimum counterexample, every vertex that is adjacent only to 2-vertices is of degree at least 7.
\begin{lemma}\label{lem:d-vertex}
    Let $G$ be a graph, let $\pi \colon E(G) \to \mathbb{F}_2$, and let $v$ be a $d$-vertex for some $d\leq 6$ adjacent to $d$ $2$-vertices in $G$. 
    Any strict $(G-N[v],\pi)$-realisation of dimension $3$ can be extended into a strict $(G,\pi)$-realisation of dimension $3$. 
\end{lemma}

\begin{proof}
    Let $N(v) = \{u_1, \dots , u_d\}$ and, for all $i\in [d]$, let $t_i$ be the neighbour of $u_i$ distinct from $v$.
    Choose $\mathbf{v} \in \mathbb{F}_2^3 \setminus\{\mathbf{0}\}$ such that $\mathbf{v} \notin \{\mathbf{t_1}, \dots , \mathbf{t_d}\}$. This is possible because there are seven non-zero vectors in $\mathbb{F}_2^3$ and $d\leq 6$.
    Then, by Lemma~\ref{lem:2-vertex}, there is a choice of $\mathbf{u}_1, \dots, \mathbf{u}_d$ that extends the realisation $(\mathbf{u})_{u \in V(G-N(u))}$ to $G$.
\end{proof}

We now turn to the 3-vertices. We first deal with the case that a 3-vertex is adjacent to two 2-vertices.
\begin{lemma}\label{lem:3-vertex}
    Let $G$ be a graph, $\pi \colon  E(G) \to \mathbb{F}_2$,  and $v$ be a $3$-vertex adjacent to two $2$-vertices $u_1$ and $u_2$ in $G$.
    Any strict $(G-\{v,u_1,u_2\},\pi)$-realisation $(\mathbf{u})_{u \in V(G-\{v,u_1,u_2\})}$ of dimension $3$ can be extended into a strict $(G,\pi)$-realisation of dimension $3$. 
\end{lemma}

\begin{proof}
    Let $t_1$ (resp. $t_2$) be the neighbour of $u_1$ (resp. $u_2$) distinct from $v$.
    By Lemma~\ref{lem:1-vertex}, there are three possible non-zero vectors $\mathbf{v}$ in $\mathbb{F}_2^3$ with which we can extend the realisation to $G-\{u_1,u_2\}$.
    So one can choose $\mathbf{v}$ among them so that $\mathbf{v}\notin \{\mathbf{t_1}, \mathbf{t_2}\}$.
    Then, by Lemma~\ref{lem:2-vertex}, there is a choice of $\mathbf{u}_1,\mathbf{u}_2$ that extends the realisation $(\mathbf{u})_{u \in V(G-\{u_1,u_2\})}$ to $G$.
\end{proof}

A {\bf deficient vertex} is a $3$-vertex having one $2$-neighbour. As indicated by Lemma~\ref{lem:3-vertex}, deficient 3-vertices are in some way critical for our discharging procedure. We give one more structural result that restricts the existence of deficient vertices. 

\begin{lemma}\label{lem:deficient}
    Let $G$ be a graph, let $\pi \colon  E(G) \to \mathbb{F}_2$.
    Let $y$ be a deficient vertex adjacent to two deficient vertices $x$ and $z$ in $G$, and let $x', y', z'$ be the $2$-neighbours of $x, y, z$ respectively.
    
    Any strict $(G-\{x,y,z, x', y', z'\},\pi)$-realisation of dimension $3$ can be extended into a strict $(G,\pi)$-realisation of dimension $3$. 
\end{lemma}

\begin{proof}
    Let $x''$ (resp. $y''$, $z''$) be the neighbour of $x'$ (resp. $y'$, $z'$) distinct from $x$ (resp. $y$, $z$).

    By Lemma~\ref{lem:2-vertex}, there are three possibilities to extend $(\mathbf{u})_{u \in V(G-\{x,y,z,x',y',z'\})}$ to $x$, and so at least two of them are different from $\mathbf{x}''$.
    Let $I_x$ be the set of these vectors. Similarly, let $I_z$ be the set of possible extensions of $(\mathbf{u})_{u \in V(G-\{x,y,z,x',y',z'\})}$ to $z$ different from $\mathbf{z}''$.

    First suppose $\pi(xy) = 1$.
    Then choose $\mathbf{x} \in I_x$ and $\mathbf{z} \in I_z$ such that $\mathbf{x} \neq \mathbf{z}$.
    Then the affine space of the common solutions $\mathbf{y}$ of the equations $\mathbf{y} \cdot \mathbf{x} = \pi(xy)$ and $ \mathbf{y} \cdot \mathbf{z} = \pi(yz)$ has dimension $1$ and does not contain $\mathbf{0}$ (since $\pi(xy) \neq 0$).
    Then choose $\mathbf{y}$ among these solutions such that $\mathbf{y} \neq \mathbf{y}''$.
    Finally, extend the realisation to $x',y',z'$ using Lemma~\ref{lem:2-vertex}.
    We proceed symmetrically if $\pi(yz) = 1$.
    Now suppose $\pi(xy)=\pi(yz)=0$.

    Suppose $I_x \cap I_z \neq \emptyset$.
    Then take $\mathbf{x} = \mathbf{z} \in I_x \cap I_z$.
    The solutions $\mathbf{y}$ of the equation $\mathbf{y} \cdot \mathbf{x} = 0$ form a space of dimension $3-1 = 2$.
    Hence there are $2^2=4$ such solutions. Then take for $\mathbf{y}$ one of these solutions distinct from $\mathbf{y}''$ and from $\mathbf{0}$.
    Then we can extend the realisation to $x',y',z'$ by Lemma~\ref{lem:2-vertex}.
    Now we assume $I_x \cap I_z = \emptyset$.

    Let $I_x = \{\mathbf{x}_1,\mathbf{x}_2\}$ and $I_z = \{\mathbf{z}_1, \mathbf{z}_2\}$.
    For every $i,j \in \{1,2\}$, the choice of $\mathbf{x} = \mathbf{x}_i$ and $\mathbf{z} = \mathbf{z}_j$ can not be extended only if
    $\{\mathbf{0}, \mathbf{y}''\} = \{\mathbf{x}_i, \mathbf{z}_j\}^\bot$.
    In particular, $\{\mathbf{x}_1, \mathbf{z}_1\}$ and $\{\mathbf{x}_1, \mathbf{z}_2\}$ span the same vector space, so $\{\mathbf{x}_1, \mathbf{z}_1,\mathbf{z}_2\}$ is linearly dependent. As $\mathbf{x}_1, \mathbf{z}_1$, and $\mathbf{z}_2$ are pairwise distinct, we obtain $\mathbf{x}_1+\mathbf{z}_1+\mathbf{z}_2=0$. A similar argument shows that   $\mathbf{x}_2+\mathbf{z}_1+\mathbf{z}_2=0$. It follows that $\mathbf{x}_1=\mathbf{x}_2$, a contradiction.
    This concludes the proof of the lemma.  
\end{proof}

\begin{proof}[Proof of Theorem~\ref{thm:3decharg}]
    We will show that for every graph with $\Mad(G)<2+\frac{8}{11}$ and for every function $\pi \colon E(G) \to \mathbb{F}_2$,
    there is a family of non-zero vectors $(\mathbf{u})_{u \in V(G)}$ in $\mathbb{F}_2^3$ such that $\pi(uv) = \mathbf{u} \cdot \mathbf{v}$ for every edge $uv$ of $G$, that is, there exists a strict $(G,\pi)$-realisation of dimension $3$.
    By Observation~\ref{obs:characterization_with_vectors}, this will imply $\diam(\mathcal{I}(G)) \leq 3$, and so prove the theorem.
    Let $G$ be a minimum counterexample to this statement.
    The above lemmas yield the following properties.
    
    \begin{itemize}
    \item[(P1)] $\delta(G)\geq 2$ (Lemma~\ref{lem:1-vertex}).
    \item[(P2)] Two $2$-vertices are not adjacent (Lemma~\ref{lem:3thread}).
    \item[(P3)] A $3$-vertex has at most one $2$-neighbour (Lemma~\ref{lem:3-vertex}).
    \item[(P4)] Every $4$- , $5$-, or $6$-vertex has at least one $(\geq 3)$-neighbour (Lemma~\ref{lem:d-vertex}).
    \item[(P5)] A deficient vertex is adjacent to at most one deficient vertex (Lemma~\ref{lem:deficient}).
    \end{itemize}
    
    We shall now use the Discharging Method.
    The initial charge of each vertex $v$  is its degree $d_G(v)$.
    We now apply the following discharging rules.
    
    \begin{itemize}
    \item[(R1)] Every $(\geq 3)$-vertex sends $\frac{4}{11}$ to its $2$-neighbours.
    \item[(R2)] Every non-deficient $(\geq 3)$-vertex sends $\frac{1}{11}$ to its deficient neighbours.
    \end{itemize}
    
    Let us examine the final charge $w(v)$ of a vertex $v$ in  $V(G)$. By (P1), we have $d_G(v)\geq 2$.
    \begin{itemize}
    \item If $v$ is a $2$-vertex, then it has two $(\geq 3)$-neighbours by (P2). It receives $\frac{4}{11}$ from each of them, so
    $w(v) = 2 + 2 \times 4/11 = 2 + 8/11$.
    \item If $v$ is a deficient $3$-vertex, then it has one $2$-neighbour and has at least one non-deficient $(\geq 3)$-neighbour by (P5).
    Thus it sends  $\frac{4}{11}$ to its $2$-neighbour and receives at least $\frac{1}{11}$ from its other neighbours.
    So $w(v) \geq 3 - 4/11 + 1/11 = 2 + 8/11$.
    \item If $v$ is a non-deficient  $3$-vertex, then its sends $1/11$ to its at most three deficient neighbours.
    So  $w(v) \geq 3 - 3\times 1/11 = 2 + 8/11$.
    \item If $v$ is a $4$-vertex, then  it sends $4/11$ to each of its $2$-neighbours by (R1) and $1/11$ to each of its deficient neighbours by (R2). By (P4), $v$ has at most three $2$-neighbours. So
    $w(v) \geq 4 - 3\times 4/11 - 1/11 = 2 + 9/11$.
    
    \item If $v$ is a $(\geq 5)$-vertex, then it sends at most $4/11$ to each of its neighbours.
    So  $w(v) \geq d(v) - d(v)\times 4/11 = d(v) \times 7/11 \geq 35/11 >  2 + 8/11$.
    \end{itemize}
    
    Hence $w(v)\geq 2 + 8/11$ for every vertex $v$.
    Consequently, $\sum_{v\in V(G)} d(v) = \sum_{v\in V(G)} w(v) \geq (2+ 8/11) |V(G)|$.
    We conclude that $\Mad(G)\geq\Ad(G) \geq 2 + 8/11$, a contradiction.
\end{proof}


\section{Complexity}\label{sec:complexity}

In this section, we consider the computational complexity of the following two related problems.

\medskip

{\sc $k$-Inversion-Distance}\\
\underline{Instance:} Two orientations $\vec{G}_1$ and $\vec{G}_2$ of a graph $G$.\\
\underline{Question:} Are $\vec{G}_1$ and $\vec{G}_2$ at distance at most $k$ in   ${\mathcal{I}}(G)$ ?

\medskip

{\sc $k$-Inversion-Diameter}\\
\underline{Instance:} A graph $G$.\\
\underline{Question:} Does ${\mathcal{I}}(G)$ have diameter at most $k$ ?

\medskip

{\sc $0$-Inversion-Distance} is trivial as two orientations of some graph are at distance $0$ if and only if they are the same. Furthermore,
{\sc $1$-Inversion-Distance} can be solved in polynomial time because of the following easy characterisation of orientations at distance $1$.

\begin{proposition}
    Let $\vec{G}_1$ and $\vec{G}_2$ be two orientations of a graph $G$.
    There exists $X$ such that $\vec{G}_2 = \Inv(\vec{G}_1 ; X)$ if and only if every edge of $E_{=}$ has at most one end-vertex incident to an edge of $E_{\neq}$.
\end{proposition}

Proposition~\ref{prop:diam1} implies that {\sc $0$-Inversion-Diameter} and {\sc $1$-Inversion-Diameter} can be solved in polynomial time.

The aim of this section is to prove that {\sc $k$-Inversion-Diameter} and {\sc $k$-Inversion-Distance} are NP-hard for all $k\geq 2$.

We need some preliminaries.
The first one is the following well-known lemma, which can be easily proved using a well-known theorem of Euler, see \cite{diestel}.

\begin{lemma}[Folklore]\label{lemma:orientation_indegree_1}
    Every graph $G$ admits an orientation $\vec{G}$ such that every vertex $v$ 
    has in-degree at least $\lfloor d(v)/2\rfloor$ in $\vec{G}$.
    Moreover, such an orientation can be computed in polynomial time.
\end{lemma}

\begin{lemma}\label{lemma:separate_different_vectors}
    Let $d \geq 2$ and let $\mathbf{u},\mathbf{v} \in \mathbb{F}_2^d \setminus \{\mathbf{0}\}$ be two distinct vectors.
    For every $x,y \in \mathbb{F}_2$, there exists $\mathbf{w} \in \mathbb{F}_2^d$ such that $\mathbf{u} \cdot \mathbf{w} = x, \mathbf{v} \cdot \mathbf{w} = y$.
\end{lemma}

\begin{proof}
    As $\mathbf{u} \neq \mathbf{v}$ are distinct non zero vectors in $\mathbb{F}_2^d$, the matrix $M$ whose rows are $u$ and $v$ has rank $2$.
    Hence its image is a $2$-dimensional space, and so the mapping $\left(\begin{array}{ccl}\mathbb{F}_2^d &\to& \mathbb{F}_2^2 \\ \mathbf{w} &\mapsto& M \cdot \mathbf{w}=(\mathbf{u} \cdot \mathbf{w}, \mathbf{v} \cdot \mathbf{w})\end{array}\right)$ is surjective.
\end{proof}

For a graph $G$, we denote by $G^{(1)}$ the graph obtained from $G$ by subdividing every edge of $G$ exactly once.
The following theorem links $\diam(\mathcal{I}(G^{(1)}))$ with $\chi(G)$. Note that similar results for $\chi_s,\chi_a,\chi_o$ have been proven by Wood~\cite{wood2005acyclic}.
\begin{theorem}\label{thm:diam_L_G_subdivided_once}
    For every graph $G$ with $\delta(G) \geq 2$ and for every integer $k \geq 2$, one can construct in polynomial time 
    two orientations $O_1,O_2$ of $G^{(1)}$ such that the following are equivalent:
    \begin{enumerate}[label=(\roman*)]
        \item $\dist_{\mathcal{I}(G^{(1)})}(O_1, O_2) \leq k$, \label{item:thm_subdivision_i}
        \item $\diam(\mathcal{I}(G^{(1)})) \leq k$, \label{item:thm_subdivision_ii}
        \item $\chi(G) \leq 2^k -1$. \label{item:thm_subdivision_iii}
    \end{enumerate} 
\end{theorem}


\begin{proof}
    Let $G$ be a graph and let $k \geq 2$ be an integer.
    We denote by $ux_{uv} v$ the path of length $2$ replacing the edge $uv \in E(G)$ in $G^{(1)}$.

    By Lemma~\ref{lemma:orientation_indegree_1}, $G$ has an orientation $\vec{G}$ of $G$ such that every vertex has in-degree at least $1$.
    For every arc $uv \in A(\vec{G})$, set $\pi_0(ux_{uv})=0$ and $\pi_0(x_{uv}v)=1$.
    Let $O_1$ be an arbitrary orientation of $G^{(1)}$ and let $O_2$ be obtained from $O_1$ by reversing every edge $e \in E(G^{(1)})$ such that $\pi_0(e)=1$.
    
    \medskip
    
    We first show that~\ref{item:thm_subdivision_iii} implies~\ref{item:thm_subdivision_ii}. Assume that $\chi(G) \leq 2^k -1$. Then there exists a mapping $\phi\colon V(G) \to \mathbb{F}_2^k \setminus \{0\}$ such that
    $\phi(u) \neq \phi(v)$ for every $uv \in E(G)$.
    Consider a function $\pi \colon E(G^{(1)}) \to \mathbb{F}_2$.
    We will construct a family $(\mathbf{w})_{w \in V(G)}$ of vectors in $\mathbb{F}_2^k$ such that $\pi(ww')=\mathbf{w} \cdot \mathbf{w}'$ for every edge $ww'$ of $G^{(1)}$.
    This will imply that $\diam(\mathcal{I}(G^{(1)})) \leq k$ by Observation~\ref{obs:characterization_with_vectors}.
    We take $\mathbf{u} = \phi(u)$ for every vertex $u \in V(G)$, and we choose $\mathbf{x}_{uv}$ such that $\mathbf{u} \cdot \mathbf{x}_{uv} = \pi(ux_{uv})$ and $\mathbf{v} \cdot \mathbf{x}_{uv} = \pi(x_{uv} v)$.
    This is possible by Lemma~\ref{lemma:separate_different_vectors}. Hence~\ref{item:thm_subdivision_ii} holds.

    \medskip

    Clearly,~\ref{item:thm_subdivision_ii} implies~\ref{item:thm_subdivision_i}.

    \medskip
    
    We finally show that~\ref{item:thm_subdivision_i} implies~\ref{item:thm_subdivision_iii}. Suppose that $\dist_{\mathcal{I}(G^{(1)})}(O_1, O_2) \leq k$.
    Then, by Observation~\ref{obs:characterization_with_vectors}, 
    there exists a family $(\mathbf{w})_{w \in V(G^{(1)})}$ of vectors in $\mathbb{F}_2^k$ such that $\pi_0(ww')=\mathbf{w} \cdot \mathbf{w}'$ for every edge $ww'$ of $G^{(1)}$.
    Then for every edge $uv$ of $G$, by the choice of $\pi_0$ we have $\mathbf{u} \cdot \mathbf{x}_{uv} \neq \mathbf{v} \cdot \mathbf{x}_{uv}$
    and so $\mathbf{u} \neq \mathbf{v}$. Moreover, for every $u \in V(G)$, $u$ is incident to an edge $e$ in $G^{(1)}$ with $\pi_0(e)=1$,
    which implies that $\mathbf{u} \neq \mathbf{0}$.
    Thus $u \mapsto \mathbf{u}$ is a proper $(2^k-1)$-colouring of $G$, yielding~\ref{item:thm_subdivision_iii}.
    
    This concludes the proof of the theorem.
\end{proof}

It is well-known that deciding whether a graph $G$ with $\delta(G)\geq 2$ has chromatic number at most $k$ is NP-complete for any fixed integer $k\geq 3$~\cite{Karp72}  and that
deciding whether a planar graph $G$ with $\delta(G)\geq 2$ has chromatic number at most $3$ is NP-complete~\cite{GLS76}. As a consequence, we obtain the following result.

\begin{corollary}\label{cor:NP-diam}
    For every integer $k \geq 2$, {\sc $k$-Inversion-Diameter} and {\sc $k$-Inversion-Distance} are NP-hard, even when restricted to graphs $G$ with $\Mad(G) < 4$.
    Moreover, {\sc $2$-Inversion-Diameter} and {\sc $2$-Inversion-Distance} are NP-hard even for planar bipartite graphs of girth at least $6$.
\end{corollary}


\section{Conclusion and further work}

In this paper, we establshed some bounds on the inversion diameter of graphs in various graph classes.
We leave as open problems to improve several of them.

\begin{problem}
    Determine $M(\Delta\leq k)$. 
\end{problem}
Theorem~\ref{thm:multipartite} and Theorem~\ref{thm:Delta} imply $k \leq M(\Delta\leq k) \leq 2 k - 1$.
Conjecture~\ref{conj:Delta} states $M(\Delta\leq k)=k$. It has been proved for $k\leq 2$. So the first open case is for $k=3$. We do not know whether
 it is true that  $\diam({\mathcal{I}}(G)) \leq 3$ for every subcubic graph $G$.  

\medskip
For any graph parameter $\gamma$, let $m(\gamma\geq k)$ be the minimum inversion diameter over all graphs $G$ with $\gamma(G)\geq k$.

\begin{problem}
    Determine $m(\delta\geq k)$, $m(\delta^*\geq k)$, and $m(\Mad\geq k)$
\end{problem}

Recall that $\delta(G) \leq \delta^*(G)\leq \Mad(G)$. Thus the fact that $\diam({\mathcal{I}}(K_k)) =  k-1$ and Theorem~\ref{theorem:high_ad_implies_high_diam} imply 
\[
    \frac{k}{2} \leq m(\Mad\geq k)\leq m(\delta^*\geq k)\leq m(\delta\geq k)\leq k.
\]
Clearly, $m(\delta\geq 0)=0$ and $m(\delta\geq 1)=1$. 
A graph $G$ with  $\delta(G) \geq 2$ contains a cycle $C$, and thus has inversion diameter at least $2$ since two orientations of $G$ which disagree on all edges of $C$ but one are at distance at least $2$ in ${\mathcal{I}}(G)$. Thus $m(\delta\geq 2)=2$. 
Theorem~\ref{thm:every_cubic_at_least_3} yields $m(\delta\geq 3)=3$.

By computer search, we found the $5$-regular graph depicted in Figure~\ref{fig:5reg_inv_diam4} which has inversion diameter $4$. Thus 
$m(\delta\geq 5)\leq 4$.
\begin{figure}[ht]
    \centering
    \begin{tikzpicture}
        \tikzstyle{vertex}=[circle,draw, top color=gray!5, bottom color=gray!30, minimum size=12pt, scale=1, inner sep=0.5pt]
        \foreach \i in {0,...,4}{
            \node[vertex] (u\i) at (-1,\i) {};
            \node[vertex] (v\i) at (1,\i) {};
            \draw (u\i) -- (v\i);
        }

        \foreach \i in {1,...,4}{
            \pgfmathtruncatemacro{\I}{\i-1}%
            \foreach \j in {0,...,\I}{
                 \pgfmathtruncatemacro{\d}{(\i-\j-1)*70/3}%
                 \draw[bend right=\d] (u\i) to (u\j);
                 \draw[bend left=\d] (v\i) to (v\j);
            }
        }
    \end{tikzpicture}
    \caption{A $5$-regular graph $G$ with inversion diameter $4$. Since this graph contains $K_5$, Theorem~\ref{thm:multipartite} implies $\diam(\mathcal{I}(G)) \geq 4$.
    The other inequality $\diam(\mathcal{I}(G)) \leq 4$ was checked by computer.}
    \label{fig:5reg_inv_diam4}
\end{figure}
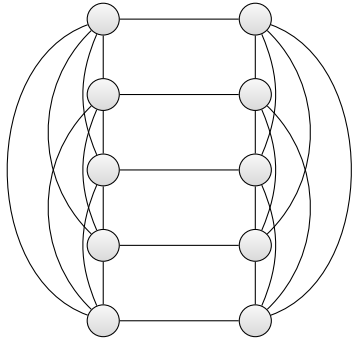

\begin{problem}  Determine $M(\tw \leq t)$. 
 \end{problem}
 
Theorem~\ref{thm:Mtw2t} and Theorem~\ref{thm:lower_bound_for_tw} yield $t+2 \leq M(\tw \leq t) \leq 2t$.
    
\begin{problem} What the maximum inversion diameter of a $K_t$-minor-free graph ? Is it in $\bigO(t\sqrt{\log t})$ ? 
\end{problem}
Corollary~\ref{cor:Kt-free} shows that it is in $\bigO(t^2)$.

\begin{problem}
    Determine $M_\mathcal{P}$ the maximum inversion diameter over all planar graphs. 
\end{problem}

Corollary~\ref{cor:planar_12} and Proposition~\ref{prop:planar5} imply $5 \leq M_\mathcal{P} \leq 12$.

\begin{problem} 
    Determine the minimum integer $g_k$ such that $\diam(\mathcal{I}(G)) \leq k$ for every planar graph of girth at least $g_k$. 
\end{problem}

Corollary~\ref{thm:planar-girth} and   Theorem~\ref{thm:planar-3} state $g_{10} \leq 5$, $g_6 \leq 7$,  and $g_3 \leq 8$.

\medskip

Particular cases mixing some of the above problems may also be studied. For example, it is already interesting to establish whether every subcubic planar graph has inversion diameter at most $3$.

\medskip

We also believe that the upper bound in Theorem~\ref{thm:multipartite} can be generalised to the blow-up of any graph as follows.
\begin{conjecture}
    For every graph $G$, for every positive integer $t$,
    \[
        \diam ({\mathcal{I}}(G[\overline{K}_t])) \leq t \cdot \diam ({\mathcal{I}}(G)).
    \] 
\end{conjecture}
 This inequality would be tight for complete graphs by Theorem~\ref{thm:multipartite}.
Note that $\diam(\mathcal{I}(G[\overline{K}_t]))$ can be characterised as in Observation~\ref{obs:characterization_with_vectors} using matrices over $\mathbb{F}_2$.
\begin{observation}
    For every graph $G$, for every positive integers $t, \ell$, the following are equivalent.
    \begin{enumerate}
        \item $\diam(\mathcal{I}(G[\overline{K}_t])) \leq \ell$.
        \item For every function $\Pi \colon E(G) \to \mathbb{F}_2^{t \times t}$, there is a family $(U)_{u \in V(G)}$ of matrices in $\mathbb{F}_2^{\ell \times t}$ such that
        \[
        \Pi(uv) = U^\top \cdot V
        \]
        for every edge $uv \in E(G)$.
    \end{enumerate}
\end{observation}

\medskip
Further, it would be interesting to understand how the density of the graph influences the diameter of its inversion graph. In particular, we are interested in the following question.
\begin{problem}
Is there an absolute constant $\alpha$ such that $\diam(\mathcal{I}(G)) \leq \alpha \sqrt{|E(G)|}$ holds for every graph $G$?
\end{problem}
\medskip

Regarding complexity, we solved the most natural questions, but many remain. One of them is the following.

\begin{problem}
    Is {\sc $2$-Inversion-Diameter} $\Pi^{\mathrm{P}}_2$-hard? 
\end{problem}

We refer the reader to \cite{ScUm02} for some $\Pi^{\mathrm{P}}_2$-complete problems.
We could also study the complexity of the following dual problem.

\medskip

{\sc $k$-Dual-Inversion-Diameter}\\
\underline{Instance:} A graph $G$.\\
\underline{Question:} Does ${\mathcal{I}}(G)$ have diameter at least $|V(G)| -1 - k$ ?
\medskip

The first question is to determine whether it is NP-complete or polynomial-time solvable for every fixed $k$.
Trivially, $0$-Dual-Inversion-Diameter is  polynomial-time solvable as the inversion diameter of a graph $G$ is at most $|V(G)| -1$.
Since an edgeless graph has inversion diameter $0$, Lemma~\ref{lem:rec-easy} implies $\diam(\mathcal{I}(G)) \leq |V(G)| - \alpha(G)$ where $\alpha(G)$ denotes the maximum size of an independent set in $G$.
Since every non-complete graph has an independent set of size $2$, it follows that a graph $G$ has inversion diameter at most $|V(G)| -1$ if and only if it is complete. Thus $1$-Dual-Inversion-Diameter is polynomial-time solvable.

\bigskip

An interesting variant of the inversion graph is the {\bf unlabelled inversion graph}, denoted by $\mathcal{U}(G)$, which is obtained from $\mathcal{I}(G)$
by identifying vertices corresponding to isomorphic orientations. Clearly, $\diam ({\cal U}(G))\leq \diam(\mathcal{I}(G))$.
Hence it would be interesting to investigate whether all the upper bounds on $\diam(\mathcal{I}(G))$ obtained in this paper can be improved for $\diam ({\cal U}(G))$. For example, the bound $\diam(\mathcal{I}(G)) \leq n -1$ is tight, but we believe it is not tight for $\diam(\mathcal{U}(G))$.
Note that $\inv(n)$ is the eccentricity of the transitive tournament in ${\cal U}(K_n)$.
It has been proved in \cite{APSSW,inversion} that $n - 2\sqrt{n\log n} \leq \inv(n)\leq n - \lceil \log (n+1) \rceil$. Together with 
Theorem~\ref{thm:multipartite}, this implies the following.

\begin{corollary}
    $n - \bigO(\sqrt{n\log n}) \leq \diam ({\cal U}(K_n)) \leq \diam ({\mathcal{I}}(K_n)) = n-1$.
\end{corollary}

However, we believe that the upper bound on $\inv(n)$  also holds for $\diam ({\cal U}(K_n))$. 
\begin{conjecture}\label{comp}
    $\diam ({\cal U}(K_n))\leq n - \lceil \log (n+1) \rceil$. 
\end{conjecture}

\section*{Acknowledgements}
This work was partially supported by the french Agence Nationale de la Recherche under contract Digraphs ANR-19-CE48-0013-01.

\bibliographystyle{alpha}
\bibliography{biblio}

\end{document}